\documentclass[11pt]{amsart}
\usepackage{graphicx}
\usepackage{amssymb}
\usepackage{amsmath}
\usepackage{amsthm}
\usepackage{mathtools}
\usepackage{tikz}
\usepackage{psfrag}
\usepackage{xcolor}

\newtheorem{thm}{Theorem}[section]
\newtheorem{lem}[thm]{Lemma}
\newtheorem{pro}[thm]{Proposition}
\newtheorem{coro}[thm]{Corollary}

\theoremstyle{definition}
\newtheorem{defn}[thm]{Definition}

\newcommand{\N}{\mathbb N}
\newcommand{\M}{\mathbb M}

\newcommand{\R}{\mathbb R}
\newcommand{\dist}{\operatorname{dist}}

\newcommand{\LL}{\mathcal L}
\newcommand{\OO}{\mathcal O}

\newcommand{\rank}{\operatorname{rank}}

\makeatletter
\numberwithin{equation}{section}
\makeatother

\begin{document}



       \author{Seonghak Kim}
       \address{Institute for Mathematical Sciences\\ Renmin University of China \\  Beijing 100872, PRC}
       \email{kimseo14@ruc.edu.cn}







       \title[Convex integration with linear constraints]{Convex integration with linear \\ constraints and its applications}


\begin{abstract}
We study solutions of the first order partial differential inclusions of the form $\nabla u\in K$, where $u:\Omega\subset\R^n\to\R^m$   and $K$ is a set of $m\times n$ real matrices, and derive a companion version to the result of {M\"uller and \v{S}ver\'ak} \cite{MSv1}, concerning a general linear constraint on the components of $\nabla u$. We then consider two applications: the vectorial eikonal equation and a $T_4$-configuration both under linear constraints. 
\end{abstract}
\maketitle

\section{Introduction}

We study the existence of solutions to the Dirichlet problem of a homogeneous partial differential inclusion
\begin{equation}\label{main-problem}
\left\{\begin{array}{ll}
  \nabla u\in K & \mbox{a.e. in $\Omega$}, \\
  u=v & \mbox{on $\partial\Omega$},
\end{array}\right.
\end{equation}
where $m,n\ge 2$ are integers, $\Omega\subset\R^n$ is a bounded domain with Lipschitz boundary, $v\in W^{1,\infty}(\Omega;\R^m)$ is a boundary map, $K$ is a subset of the space $\M^{m\times n}$ of $m\times n$ real matrices, and $u\in v+ W_0^{1,\infty}(\Omega;\R^m)$ is a solution to the problem.

Such a problem of differential inclusion (\ref{main-problem})   has stemmed from the study of models of crystal microstructure by {Ball and James} \cite{BJ, BJ1} and {Chipot and Kinderlehrer} \cite{CK}. Later, {M\"uller and \v{S}ver\'ak} \cite{MSv1, MSv2} generalized the theory of convex integration of Gromov \cite{Gr} and applied the results to the two-well problem in the theory of martensite  \cite{MSv1} and to the construction of \emph{wild} solutions of some $2\times 2$ elliptic system \cite{MSv2}. Constructing a suitable in-approximation and applying the result of \cite{MSv1}, {Conti, Dolzmann and Kirchheim} \cite{CDK} obtained Lipschitz minimizers for the three-well problem in solid-solid phase transitions. On the other hand, {Dacorogna and Marcellini} \cite{DM1} and {Dacorogna and Tanteri} \cite{DT} extensively studied (\ref{main-problem}) and its inhomogeneous version under the Baire category framework.

The generalization of Gromov's result by {M\"uller and \v{S}ver\'ak} \cite{MSv1} was pursued in two directions. Firstly, they showed that constraints on a minor of $\nabla u$ can be imposed in solving problem (\ref{main-problem}) under the convex integration method. Secondly, they enlarged the set of matrices, in which $\nabla v$ can stay  for solvability of (\ref{main-problem}), from the lamination convex hull of $K$ to its rank-one convex hull when $K$ is open and bounded. Also, an in-approximation scheme was adopted to handle the case that $K$ is not necessarily open.

In this paper, we show that one can impose a \emph{general} linear constraint on the components of $\nabla u$ to solve problem (\ref{main-problem}) in the spirit of \cite{MSv1} and provide two examples of application: the vectorial eikonal equation and a $T_4$-configuration both under linear constraints.  
Unlike \cite{MSv1}, we avoid using  piecewise linear approximation for rank-one connections, but instead maintain $C^1$ regularity in our approximation. This turns out to be possible in case of a linear constraint (also in the unconstrained case) although in the special case that $m=n\ge 2$ with the  constraint $\mathrm{div}u=\mathrm{const}$, piecewise linear approximation can be constructed as mentioned in \cite{MSv1} and proved in \cite{Po1}.

To state our main results, we first introduce some definitions.
A set $E\subset\M^{m\times n}$ is called \emph{lamination convex} if $[\xi_1,\xi_2]\subset E$ for all $\xi_1,\xi_2\in E$ with $\rank(\xi_1-\xi_2)=1$, where $[\xi_1,\xi_2]$ denotes the closed line segment in $\M^{m\times n}$ joining $\xi_1$ and $\xi_2$. The \emph{lamination convex hull} $E^{lc}$ of a set $E\subset\M^{m\times n}$ is defined to be the intersection of all lamination convex sets in $\M^{m\times n}$ containing $E$; that is, it is the smallest lamination convex set in $\M^{m\times n}$ containing $E$.
A function $f:\M^{m\times n}\to \R$ is called \emph{rank-one convex} if
\[
f(\lambda\xi_1+(1-\lambda)\xi_2)\le\lambda f(\xi_1)+(1-\lambda) f(\xi_2)
\]
for all $\xi_1,\xi_2\in\M^{m\times n}$ with $\rank(\xi_1-\xi_2)=1$ and all $\lambda\in[0,1]$, or equivalently, if $\R\ni s\mapsto f(\xi+s a\otimes b)$ is convex for each $(\xi,a,b)\in\M^{m\times n}\times\R^m\times\R^n$.  The \emph{rank-one convex hull} $K^{rc}$ of a compact set $K\subset\M^{m\times n}$ is defined as
\[
K^{rc}=\Big\{\xi\in\M^{m\times n}\,|\, f(\xi)\le\sup_{K} f\;\;\forall f:\M^{m\times n}\to \R\;\;\mbox{rank-one convex}\Big\}.
\]
The \emph{rank-one convex hull} $E^{rc}$ of a  set $E\subset\M^{m\times n}$ is then defined to be
\[
E^{rc}=\bigcup\{K^{rc}\,|\,K\subset E,\,\mbox{$K$ is compact}\}.
\]
With this definition, the rank-one convex hull $V^{rc}$ of any open set $V$ in $\M^{m\times n}$ is again open in $\M^{m\times n}.$

Throughout the paper, we reserve the following notations unless otherwise stated. Let $m,n\ge 2$ be integers, and let $\Omega\subset\R^n$ be a bounded domain with Lipschitz boundary. Let $L\in\M^{m\times n}\setminus\{0\}$, and its corresponding linear function $\LL:\M^{m\times n}\to\R$ is given by
\[
\LL(\xi)=L\cdot\xi=\sum_{1\le i\le m,\,1\le j\le n} L_{ij}\xi_{ij}\quad\forall\xi\in\M^{m\times n}.
\]
As an abuse of notation, we often view $L$ as the linear map $b\mapsto Lb$ from $\R^n$ into $\R^m$, which should be distinguished from $\mathcal{L}$.
We fix any number $t\in\R$ and write
\[
\Sigma_t=\{\xi\in\M^{m\times n}\,|\,\LL(\xi)=t\},
\]
which is an $(mn-1)$-dimensional flat manifold in $\M^{m\times n}$.  We denote by $\partial|_{\Sigma_t}$ the relative boundary in the space $\Sigma_t$.

A map $v:\Omega\to\R^m$ is called \emph{piecewise} $C^1$  if there exists a sequence $\{\Omega_j\}_{j\in\N}$ of disjoint open subsets of $\Omega$ whose union has measure $|\Omega|$ and such that $v\in C^1(\bar{\Omega}_j;\R^m)$ for all $j\in\N$.

We now state the first main result of the paper  as follows.

\begin{thm}\label{thm:main-1}
Assume
\begin{equation}\label{assume-1}
Lb\ne 0\in\R^m\quad \forall b\in\R^n\setminus\{0\}.
\end{equation}
Let $U$ be a bounded open set in $\Sigma_t$, and let $v\in W^{1,\infty}(\Omega;\R^m)$ be a piecewise $C^1$ map satisfying
\[
\nabla v\in U^{rc}\quad\mbox{a.e. in $\Omega$}.
\]
Then for each $\epsilon>0$, there exists a map $u\in W^{1,\infty}(\Omega;\R^m)$ such that
\[
\left\{\begin{array}{l}
  \nabla u\in U \quad \mbox{a.e. in $\Omega$}, \\
  u=v \quad \mbox{on $\partial\Omega$}, \\
  \|u-v\|_{L^\infty(\Omega)}<\epsilon.
\end{array}\right.
\]
\end{thm}
Note that hypothesis (\ref{assume-1}) is equivalent to saying that $m\ge n$ and  the linear map $L:\R^n\to\R^m$ is injective.

To deal with the sets that may  not be open, we adopt the following notion from \cite{MSv1}.

\begin{defn}
Let $K\subset\Sigma_t$. A sequence $\{U_j\}_{j\in\N}$ of open sets in $\Sigma_t$ is called an \emph{in-approximation} of $K$ in $\Sigma_t$ if the following are satisfied:
\begin{itemize}
\item[(i)] $U_j$ $(j\in\N)$ are uniformly bounded,
\item[(ii)] $U_j\subset U_{j+1}^{rc}$ for every $j\in\N$, and
\item[(iii)] $U_j\to K$ as $j\to \infty$ in the following sense: If $\xi_i\in U_i$ for all $i\in\N$ and $\xi_j\to\xi$ as $j\to\infty$ for some $\xi\in\Sigma_t$, then $\xi\in K$.
\end{itemize}
\end{defn}

The second main result of this paper is then formulated as follows.

\begin{thm}\label{thm:main-2}
Assume \emph{(\ref{assume-1})} and $K\subset \Sigma_t$. Let $\{U_j\}_{j\in\N}$ be an in-approxima-tion of $K$ in $\Sigma_t$, and let $v\in W^{1,\infty}(\Omega;\R^m)$ be a piecewise $C^1$ map satisfying
\[
\nabla v\in U_1\quad\mbox{a.e. in $\Omega$}.
\]
Then for each $\epsilon>0$, there exists a map $u\in W^{1,\infty}(\Omega;\R^m)$ such that
\[
\left\{\begin{array}{l}
  \nabla u\in K\quad\mbox{a.e. in $\Omega$}, \\
  u=v\quad\mbox{on $\partial\Omega$}, \\
  \|u-v\|_{L^\infty(\Omega)}<\epsilon.
\end{array}\right.
\]
\end{thm}

The first  application of our results concerns the Drichlet problem of the vectorial eikonal equation
\begin{equation}\label{eikonal-problem-main}
\left\{
\begin{array}{ll}
  |\nabla u|=1 & \mbox{a.e. in $\Omega$}, \\
  u=v_{\eta,\gamma} & \mbox{on $\partial\Omega$},
\end{array}\right.
\end{equation}
where $\eta\in\M^{m\times n}$, $\gamma\in\R^m$ and $v_{\eta,\gamma}(x):=\eta x+\gamma$ $(x\in\Omega)$. Here, we look for solutions $u$ in the space $v_{\eta,\gamma}+W_0^{1,\infty}(\Omega;\R^m).$  If $u$ is a solution to problem (\ref{eikonal-problem-main}), then
\[
|\eta||\Omega|=\bigg|\int_\Omega\eta dx\bigg|=\bigg|\int_\Omega\nabla u(x)dx\bigg|\le |\Omega|;
\]
so $|\eta|\le1$. Thus there is no solution to (\ref{eikonal-problem-main}) if $|\eta|>1$. If $|\eta|=1$, we have the trivial solution $u= v_{\eta,\gamma}$ to (\ref{eikonal-problem-main}). So we assume  $|\eta|<1$. In case of $m=n=1$ with $\Omega=(0,1)$, one can trivially construct infinitely many solutions $u$ to (\ref{eikonal-problem-main}) whose graph has slopes $\pm1$ a.e. in $\Omega$, left-end point $(0,\gamma)$ and right-end point $(1,\eta+\gamma)$. Motivated by this simplest case, we may pose a question: For the vectorial case $m,n\ge 2$, when  $\eta^\pm\in\M^{m\times n}$ are two distinct matrices with $|\eta^\pm|=1$, is there a  solution $u$ to (\ref{eikonal-problem-main}) which assumes only the two gradient values $\eta^\pm$ a.e. in $\Omega$? The answer is negative when $\mathrm{rank}(\eta^+ -\eta^-)\ge 2$ due to the rigidity of the two gradient problem \cite{BJ}. A \emph{partially} positive answer is available when $\mathrm{rank}(\eta^+ -\eta^-)=1$ and $\eta\in(\eta^+ ,\eta^-)$. In this case, one can employ either the convex integration method with an in-approximation scheme \cite{MSv1} or the Baire category method \cite{DM1} to obtain infinitely many solutions $u$ to (\ref{eikonal-problem-main}) such that $\mathrm{dist}(\nabla u,\{\eta^+,\eta^-\})<\epsilon$  a.e. in  $\Omega$, for any given $\epsilon>0$. We can even impose suitable linear constraints as follows.

\begin{thm}\label{thm:eikonal}
Suppose $\eta\in\M^{m\times n}$, $|\eta|<1$,  $\gamma\in\R^m$, $(a,b)\in\R^m\times\R^n$, $a\neq 0$, $Lb\neq 0$, $\mathcal{L}(a\otimes b)=0$ and  $\epsilon>0$. Then there are infinitely many maps $u\in W^{1,\infty}(\Omega;\R^m)$ satisfying
\begin{equation}\label{eikonal-problem}
\left\{
\begin{array}{ll}
  |\nabla u|=1 & \mbox{a.e. in $\Omega$}, \\
  \mathcal{L}(\nabla u)=t & \mbox{a.e. in $\Omega$}, \\
  \mathrm{dist}(\nabla u,\{\eta^\pm_{a\otimes b}\})<\epsilon & \mbox{a.e. in $\Omega$}, \\
  u=v_{\eta,\gamma} & \mbox{on $\partial\Omega$}, \\
  \|u-v_{\eta,\gamma}\|_{L^\infty(\Omega)}<\epsilon, &
\end{array}\right.
\end{equation}
where $v_{\eta,\gamma}(x):=\eta x+\gamma$ $(x\in\Omega)$, $t:=\mathcal{L}(\eta)$, and $s^+>0>s^-$ are the unique numbers, with $\eta^\pm_{a\otimes b}:=\eta+s^\pm a\otimes b$, such that $|\eta^\pm_{a\otimes b}|=1$.
\end{thm}

Note that the constraint $\mathcal{L}(\nabla u)=t$ (i.e., $\nabla u\in\Sigma_t$) restricts the selection of a rank-one direction $a\otimes b$ for lamination as $\mathcal{L}(a\otimes b)=L\cdot(a\otimes b)=0$. This is inevitable in the convex integration method since the gradient of a map involved in approximation should always stay in the manifold of constraint $\Sigma_t$.  As an application of Theorem \ref{thm:main-2}, Theorem \ref{thm:eikonal} can be proved directly by constructing a simple in-approximation in $\Sigma_t$  when $m\ge n\ge2$ and the linear map $L:\R^n\to\R^m$ is injective. This additional hypothesis arises due to the general feature of a differential inclusion in Theorem \ref{thm:main-2} that does not single out a \emph{principal} rank-one direction $a\otimes b$ for lamination. In Section \ref{sec:proof-applications}, we explain the use of such an in-approximation for  the special case of Theorem \ref{thm:eikonal} in terms of Theorem \ref{thm:main-2} and also provide the complete proof of Theorem \ref{thm:eikonal} under the Baire category framework.

The other application focuses on a $T_4$-configuration (see \cite{MSv2} for precise definition). Consider the set $K\subset \M^{2\times 2}_{diag}$ consisting of the four matrices
\begin{equation}\label{app2-1-K-set-definition}
A_1=-A_3=\begin{pmatrix}
3 & 0 \\
0 & -1
\end{pmatrix},\;\;
A_2=-A_4=\begin{pmatrix}
1 & 0 \\
0 & 3
\end{pmatrix},
\end{equation}
where $\M^{2\times 2}_{diag}$ denotes the space of $2\times 2$ diagonal matrices. Such a set $K$ was discovered independently by \cite{Sc, AH, Ca, Ta1} as an example of a compact set $K$ for which $K^{lc}\neq K^{rc}$  and has found striking applications for constructing wild solutions in elliptic system \cite{MSv2}, parabolic system \cite{MRS},  porous media equation \cite{CFG} and active scalar equations \cite{Sy}. Actually, it is easy to check that $K^{rc}$ contains the segments $[A_1,J_2]$, $[A_2,J_3]$, $[A_3,J_4]$, $[A_4,J_1]$ and the convex hull of $\{J_1,J_2,J_3,J_4\}$, where $J_1=-J_3=\mathrm{diag}(-1,-1)$ and $J_2=-J_4=\mathrm{diag}(1,-1)$. On the other hand, since $K$ has no rank-one connection, we simply have $K^{lc}=K$. By the same reason,  the differential inclusion $\nabla u\in K$ only admits the trivial solutions $\nabla u=A_i$ $(i=1,2,3,4)$ due to the rigidity of the four gradient problem \cite{CK1}. Regardless of such rigidity, it is still possible to have the gradient $\nabla u$ concentrated near the matrices $A_1,A_2,A_3,A_4$ by a nontrivial  map $u\in \eta x+W^{1\infty}_0(\Omega;\R^2)$ if $\eta\in K^{rc}$ (see \cite[Corollary 1.5]{MSv1}). 

We can slightly improve this corollary by imposing linear constraints as follows.

\begin{coro}\label{coro:T4-configuration}
Let $\Omega\subset\R^2$ be a bounded domain with Lipschitz boundary, and let $k\in\R\setminus\{0\}$ and $L=\begin{pmatrix}
0 & k \\
1 & 0
\end{pmatrix}$. Let $K\subset\M^{2\times 2}_{diag}$ be the set consisting of the four matrices in (\ref{app2-1-K-set-definition}), and let $\eta\in K^{rc}$, $\gamma\in \R^2$ and $\epsilon>0$. Then there exists a map $u\in W^{1,\infty}(\Omega;\R^2)$ such that
\[
\left\{
\begin{array}{ll}
  \mathrm{dist}(\nabla u,K)<\epsilon & \mbox{a.e. in $\Omega$}, \\
  \mathcal{L}(\nabla u)=0& \mbox{a.e. in $\Omega$}, \\
  u=v_{\eta,\gamma} & \mbox{on $\partial\Omega$,} \\
  \|u-v_{\eta,\gamma}\|_{L^\infty(\Omega)}<\epsilon, &
\end{array}\right.
\]
where $v_{\eta,\gamma}(x):=\eta x+\gamma$ $(x\in\Omega)$.
\end{coro}

Here the constraint $\mathcal{L}(\nabla u)=0$ also reads as $\partial_{x_1}u^2+k\partial_{x_2}u^1=0$ or as $\nabla u\in \Sigma_0$.
Observe that the dimensions of $\M^{2\times 2}_{diag}\subsetneq\Sigma_0\subsetneq \M^{2\times 2}$ are 2, 3 and 4, respectively. Targeting the set $K\subset \M^{2\times 2}_{diag}$, the gradient $\nabla u$ of our \emph{approximate} solution $u$ to the differential inclusion $\nabla u\in K$ may go beyond the plane $\M^{2\times 2}_{diag}$ but always stays in the 3-dimensional  manifold $\Sigma_0$. Note also that $\Sigma_0$ is the space of $2\times 2$ symmetric matrices for $k=-1$ and that of $2\times 2$ skew-symmetric matrices if $k=1$.


The rest of the paper is organized as follows. Section \ref{sec:rank-one} concerns a functional tool, Theorem \ref{thm:main-lemma}, for the passage from lamination convex hull to rank-one convex hull in the manifold of constraint $\Sigma_t$.
In Section \ref{sec:rank-1}, we equip with the main tool, Theorem \ref{thm:rank-1}, for rank-one smooth approximation under a general linear constraint that eventually leads, with the help of Theorem \ref{thm:main-lemma}, to Lemma \ref{lem:pre-result}, which is a precursor to the main results of the paper, Theorems \ref{thm:main-1} and \ref{thm:main-2}. Then the proof of Theorems \ref{thm:main-1} and \ref{thm:main-2} is provided in Section \ref{sec:proof-main-theorems}. Section \ref{sec:proof-applications} finishes the proof of the applications, Theorem \ref{thm:eikonal} and Corollary \ref{coro:T4-configuration}. Lastly, the proof of Theorem \ref{thm:rank-1} is included in Section \ref{sec:proof-rank-1}.

In closing this section, we add some notations. For a vector $a\in\R^n$, we write $|a|=(\sum_{i}a_i^2)^{1/2}$ for its Euclidean norm. For a matrix $\xi\in\M^{m\times n}$, we denote by $|\xi|=(\sum_{i,j}\xi_{ij}^2)^{1/2}$ the Hilbert-Schmidt norm of $\xi$. For a measurable set $E\subset\R^n$, its Lebesgue measure is denoted by $|E|.$ Some other notations will be introduced as we go along the paper if necessary.

\section{Rank-one convex functions and hulls}\label{sec:rank-one}
This section prepares a powerful tool that enables us to handle rank-one convex hulls instead of lamination convex hulls, which can be strictly smaller than the former ones. Our version of such a tool, Theorem \ref{thm:main-lemma}, for the manifold of constraint $\Sigma_t$ originates from   \cite[Theorem 3.1]{MSv1} that is dealing with the case of a constraint on a minor of $\nabla u$ and unconstrained case and that was motivated by and generalized from a result of \cite{Pe1}. We thus closely follow the exposition and relevant proofs from Section 3 of \cite{MSv1} but add more details for the reader's convenience.

We first introduce many definitions. Let $\mathcal O$ be an open set in $\M^{m\times n}$ or in $\Sigma_t$. A function $f:\mathcal O\to \R$ is called \emph{rank-one convex} if, for all $\xi_1,\xi_2\in\mathcal O$ with $[\xi_1,\xi_2]\subset\mathcal O$ and $\rank(\xi_1-\xi_2)=1,$ we have
\[
f(\lambda\xi_1+(1-\lambda)\xi_2)\le\lambda f(\xi_1)+(1-\lambda)f(\xi_2)
\]
for all $\lambda\in[0,1].$

We denote by $\mathcal P$ the set of all probability (Borel) measures on $\M^{m\times n}$ with compact support. For a compact set $K\subset\M^{m\times n}$, let $\mathcal{P}(K)$ denote the set of all $\nu\in\mathcal{P}$ with $\mathrm{supp}(\nu)\subset K$. For each $\nu\in\mathcal P$, we write its center of mass as $\bar\nu=\int_{\M^{m\times n}}\xi d\nu(\xi)\in\M^{m\times n}$.

Let $\nu\in\mathcal P$. For a continuous function $f:\M^{m\times n}\to\R$, we write $\langle\nu, f\rangle=\int_{\M^{m\times n}}f(\xi)d\nu(\xi)\in\R$.
We say that $\nu$ is a \emph{laminate} if $\langle \nu,f\rangle \ge f(\bar\nu)$ for every rank-one convex function $f:\M^{m\times n}\to\R$. For a  compact set $K\subset\M^{m\times n},$ we define
\[
\mathcal M^{rc}(K)=\{\nu\in\mathcal P(K)\,|\,\mbox{$\nu$ is a laminate}\}.
\]
By \cite{Pe1}, it has been well known that $K^{rc}=\{\bar\nu\,|\,\nu\in\mathcal{M}^{rc}(K)\}.$

Let $\xi\in\M^{m\times n}$, and let $\delta_{\xi}$ denote the Dirac mass at $\xi$. As $\mathrm{supp}(\delta_{\xi})=\{\xi\}$, we have $\delta_{\xi}\in\mathcal P$. Also, $\bar{\delta}_{\xi}=\xi.$ Thus for any continuous function $f:\M^{m\times n}\to\R$, we have $\langle \delta_\xi,f\rangle=f(\xi)=f(\bar\delta_\xi)$; hence $\delta_\xi$ is a laminate.

Let $\OO$ be an open set in $\Sigma_t$.
We define a class $\mathfrak{L}(\OO)$ of laminates in $\mathcal{P}$, called \emph{laminates of finite order} in $\OO$, inductively as follows:
\begin{itemize}
\item[(i)] For each $\xi\in\OO$, we have $\delta_\xi\in\mathfrak L(\OO)$, called a \emph{laminate of order} $0$.
\item[(ii)] If $k\in\N$, $\lambda_1,\cdots,\lambda_k> 0$, $\sum_{j=1}^k\lambda_j=1$, $\xi_1,\cdots,\xi_k\in\OO$ are pairwise distinct, $\sum_{j=1}^k\lambda_j \delta_{\xi_j}\in\mathfrak L(\OO)$ is a laminate of order $k-1$, and $\xi_k=s\eta_1+ (1-s)\eta_2$ for some $0< s< 1$ and some rank-one segment $[\eta_1,\eta_2]\subset\OO$ with $\eta_i\neq \xi_j$ $(i=1,2,\, j=1,\cdots,k-1)$, then $\sum_{j=1}^{k-1}\lambda_j \delta_{\xi_j}+s\lambda_k\delta_{\eta_1}+(1-s)\lambda_k\delta_{\eta_2}\in\mathfrak L(\OO)$, called a \emph{laminate of order} $k$.
\end{itemize}
From definition, it follows that $\langle\nu, f\rangle\ge f(\bar\nu)$ for each $\nu\in\mathfrak{L}(\mathcal{O})$ and each rank-one convex function $f:\mathcal{O}\to\R.$

Let $K\subset \M^{m\times n}$ be a compact set. From  the definition of $K^{rc}$, for any  $\xi\in\M^{m\times n}$, we have that  $\xi\not\in K^{rc}$ if and only if $f(\xi)>0$ for some rank-one convex function $f:\M^{m\times n}\to\R$ with $f\le 0$ on $K$. Now, let $K$ be a compact subset of $\Sigma_t$. The \emph{rank-one convex hull} $K^{rc,\Sigma_t}$ of $K$ \emph{relative to} $\Sigma_t$ is a subset of $\Sigma_t$ defined as follows: For each $\xi\in\Sigma_t$, $\xi\not\in K^{rc,\Sigma_t}$ if and only if $f(\xi)>0$ for some rank-one convex function $f:\Sigma_t\to\R$ with $f\le 0$ on $K$. Then the \emph{rank-one convex hull} $E^{rc,\Sigma_t}$ of a set $E\subset\Sigma_t$ \emph{relative to} $\Sigma_t$ is defined to be
\[
E^{rc,\Sigma_t}=\bigcup\{K^{rc,\Sigma_t}\,|\,K\subset E,\,\mbox{$K$ is compact}\}.
\]
With this definition, the rank-one convex hull $V^{rc,\Sigma_t}$ of any open set $V$ in $\Sigma_t$ relative to $\Sigma_t$ is also open in $\Sigma_t.$
Another simple fact that is needed later is as follows.
\begin{pro}\label{prop:rank-one-1}
Let $K$ be a compact subset of $\Sigma_t.$ Then $K^{rc,\Sigma_t}$ and $K^{rc}$ are both compact, and
\[
K^{rc,\Sigma_t}\subset K^{rc}\subset\Sigma_t.
\]
\end{pro}

\begin{proof}
Compactness of $K^{rc,\Sigma_t}$ and $K^{rc}$ easily follows from the definitions.

To show that $K^{rc}\subset\Sigma_t$, choose an open ball $B$ in $\M^{m\times n}$ with $K\subset B$, and  define $f(\xi)=\mathrm{dist}(\xi,\bar{B}\cap\Sigma_t)$ for all $\xi\in\M^{m\times n}.$ As $\bar{B}\cap\Sigma_t$ is convex and compact, the function $f:\M^{m\times n}\to\R$ is (rank-one) convex and satisfies that $f(\xi)>0$ for all $\xi\in \M^{m\times n}\setminus(\bar{B}\cap\Sigma_t)$. Since $f=0$ on $K$, we have from definition that $\xi\not\in K^{rc}$ for all $\xi\in \M^{m\times n}\setminus(\bar{B}\cap\Sigma_t)$; thus $K^{rc}\subset\Sigma_t$.

Next, let $\xi\in\Sigma_t\setminus K^{rc}$. By definition, we have $g(\xi)>0$ for some rank-one convex function $g:\M^{m\times n}\to\R$ with $g\le 0$ on $K$. As $g|_{\Sigma_t}:\Sigma_t\to\R$ is rank-one convex, we have $\xi\not\in K^{rc,\Sigma_t}.$ Thus $K^{rc,\Sigma_t}\subset K^{rc}$.
\end{proof}

We now state the main result of this section as follows.

\begin{thm}\label{thm:main-lemma}
Let $K$ be a compact subset of $\Sigma_t$. Then $K^{rc}=K^{rc,\Sigma_t}$. Let $\nu\in\mathcal M^{rc}(K)$, and let $\OO$ be an open set in $\Sigma_t$ containing $K^{rc}$. Then there exists a sequence $\nu_j\in\mathfrak L(\OO)$ with $\bar\nu_j=\bar\nu$ that converges weakly* to $\nu$ in $\mathcal P$.
\end{thm}

An immediate consequence of this theorem is as follows.

\begin{coro}\label{prop:rank-1-hull-open}
If $V$ is an open set in $\Sigma_t$, then $V^{rc}=V^{rc,\Sigma_t}$, and $V^{rc}$ is open  in $\Sigma_t$.
\end{coro}



From this corollary, we have the following.

\begin{coro}\label{prop:sets}
Let $U$ be a bounded open set in $\Sigma_t$. Then for any compact set $K\subset U^{rc}$, there exists an open set $V$ in $\Sigma_t$ with $\bar V\subset U$ such that
\[
K\subset V^{rc}\quad\mbox{and}\quad \overline{({V}^{rc})}\subset U^{rc}.
\]
\end{coro}

\begin{proof}
For each $\epsilon>0$, let
\[
U_\epsilon=\{\xi\in U\,|\,\mathrm{dist}(\xi,\partial|_{\Sigma_t}U)>\epsilon\}.
\]
By the definition of $U^{rc}$, we easily see that
\[
U^{rc}=\cup_{\epsilon>0}(\bar{U}_\epsilon)^{rc}=\cup_{\epsilon>0}U_\epsilon^{rc}.
\]
Let $K\subset U^{rc}$ be any compact set; then we  have $K\subset U_{\epsilon_0}^{rc}$ for some $\epsilon_0>0$ since $\{U_\epsilon^{rc}\}_{\epsilon>0}$ is an open covering for $K$ by Corollary \ref{prop:rank-1-hull-open}. We set $V=U_{\epsilon_0}$. By definition, we now have
\[
\overline{(V^{rc})}\subset (\bar{V})^{rc}\subset U^{rc}.
\]
\end{proof}

We now pay attention to the proof of Theorem \ref{thm:main-lemma} that requires  two ingredients. The first one is on the representation of the \emph{rank-one convex envelope} of a continuous function in terms of laminates of finite order.

\begin{lem}\label{lem:rank-one-envelope}
Let $\OO$ be an open set in $\Sigma_t$, and let $f:\OO\to\R$ be a continuous function. Let $R_\OO f:\OO\to\R\cup\{-\infty\}$ be the function defined by
\[
R_\OO f(\xi)=\sup\big\{g(\xi)\,|\, \mbox{$g:\OO\to\R$ is rank-one convex, $g\le f$ in $\OO$} \big\}
\]
for all $\xi\in\OO$. Assume that there exists a rank-one convex function $g_0:\OO\to\R$ with $g_0\le f$ in $\OO$. Then for each $\xi\in\OO$, we have
\[
R_\OO f(\xi)=\inf\big\{\langle\nu,f\rangle\,|\, \nu\in\mathfrak{L}(\OO),\;\bar\nu=\xi\big\}\in\R.
\]
\end{lem}

\begin{proof}
Let $\tilde f:\OO\to\R\cup\{-\infty\}$ be the function given by
\[
\tilde f(\xi)=\inf\big\{\langle\nu,f\rangle\,|\, \nu\in\mathfrak{L}(\OO),\;\bar\nu=\xi\big\}\quad\forall \xi\in\OO.
\]
We have to show that $-\infty<\tilde f=R_{\mathcal{O}}f$ in $\mathcal O$.

Let $g:\OO\to\R$ be any rank-one convex function with $g\le f$ in $\OO$. Fix any $\xi\in\OO$, and let $\nu\in\mathfrak L(\OO)$ and $\bar\nu=\xi$. Then $\langle \nu, f\rangle\ge\langle\nu,g\rangle\ge g(\bar\nu)=g(\xi)$. Taking supremum on such $g$'s and infimum on such $\nu$'s, we have $-\infty<R_\OO f(\xi)\le \tilde f(\xi)$ from the existence of the function $g_0$. In particular, $\tilde f$ is real-valued in $\OO$.

Let us now check that $\tilde f:\OO\to\R$ is rank-one convex and that $\tilde f\le f$ in $\OO$. Once these are done, it follows from the definition of  $R_\OO f$ that $\tilde f\le R_\OO f$ in $\OO$; thus $R_\OO f=\tilde f$ in $\OO$, and the proof is complete.

We now turn to the remaining assertions.
Let $\xi\in\OO$. As $\delta_\xi\in\mathfrak{L}(\OO)$ and $\bar\delta_\xi=\xi$, the definition of $\tilde f$ implies $\tilde f(\xi)\le\langle \delta_\xi,f\rangle=f(\xi)$; thus, $\tilde f\le f$ in $\OO$. It remains to show that $\tilde f:\OO\to\R$ is rank-one convex. Let $[\xi_1,\xi_2]\subset\OO$ be any rank-one segment. Let $0< s< 1$ and $\xi=s\xi_1+(1-s)\xi_2\in\OO$. Let $\mu,\nu\in\mathcal{L}(\OO)$ be such that $\bar\mu=\xi_1$, $\bar\nu=\xi_2$. As $\rank(\xi_1-\xi_2)=1$, it is easy to see (by double induction) that $s\mu+(1-s)\nu\in\mathfrak L(\OO)$ and $\overline{s\mu+(1-s)\nu}=\xi$. So, by the definition of $\tilde f$, we get $\tilde f(\xi)\le \langle s\mu+(1-s)\nu,f \rangle=s\langle \mu,f \rangle+(1-s)\langle\nu,f\rangle$. Taking infimum on such $\mu$'s and $\nu$'s, we get $\tilde f(s\xi_1+(1-s)\xi_2)=\tilde f(\xi)\le s\tilde f(\xi_1)+(1-s)\tilde f(\xi_2).$ Thus, $\tilde f:\OO\to\R$ is rank-one convex.
\end{proof}

The other lemma for the proof of Theorem \ref{thm:main-lemma} is stated as follows, which may not be so simple to verify.

\begin{lem}\label{lem:main-approximation}
Let $K$ be a compact subset of $\Sigma_t$, and let $\OO$ be an open set in $\Sigma_t$ containing $\tilde K:=K^{rc,\Sigma_t}$. Let $f:\OO\to\R$ be a rank-one convex function. Then for each $\epsilon>0$, there exists a rank-one convex function $F:\M^{m\times n}\to\R$ such that $|F-f|<\epsilon$ on $\tilde K$.
\end{lem}

To prove this lemma, we need some auxiliary results. We begin with a simple observation on rank-one convex functions.

\begin{lem}\label{lem:rank-1-property}
Let $U,\,V$ be open sets in $\Sigma_t$ with $\bar V\subset U$, and let $f_1:\Sigma_t\to\R$ and $f_2:U\to\R$ be rank-one convex functions such that $f_2\ge f_1$ in $V$ and that $f_1=f_2$ on $\partial|_{\Sigma_t} V$. Let $f:\Sigma_t\to\R$ be the function given by
\[
f(\xi)=\left\{\begin{array}{ll}
                f_2(\xi), & \xi\in V, \\
                f_1(\xi), & \xi\in \Sigma_t\setminus V.
              \end{array}
 \right.
 \]
Then $f:\Sigma_t\to\R$ is rank-one convex. Moreover, the same result holds when $\Sigma_t$ is replaced by $\M^{m\times n}$.
\end{lem}

As a precursor of this lemma, we first deal with its one-dimensional version.

\begin{lem}\label{lem:rank-1-property-1D}
Let $\{I_j\}_{j\in J}$ be a countable collection of disjoint  open intervals in $(0,1)$. Let $f_1:[0,1]\to\R$ be a convex function, and let $f_2:\cup_{j\in J}\bar{I}_j\to\R$ be a function such that it is convex on each interval $\bar{I}_j$, $f_2\ge f_1$ in $\cup_{j\in J}I_j$, and $f_1=f_2$ on $\cup_{j\in J}\partial I_j$. Define
\[
f(x)=\left\{\begin{array}{ll}
                f_2(x), & x\in \cup_{j\in J}I_j, \\
                f_1(x), & x\in [0,1]\setminus \cup_{j\in J}I_j.
              \end{array}
 \right.
 \]
Then $f:[0,1]\to\R$ is convex.
\end{lem}

\begin{proof}
Let $x_0<x_1$ be any two numbers in $[0,1]$, and let $\bar x=\lambda x_0+(1-\lambda)x_1$, where $0<\lambda<1$ is any fixed number. We have to show that
\begin{equation}\label{convex-inqeuality}
f(\bar x)\le \lambda f(x_0)+(1-\lambda)f(x_1).
\end{equation}

If $\bar x\in[0,1]\setminus\cup_{j\in J}I_j$, then
\[
f(\bar x)=f_1(\bar x)\le \lambda f_1(x_0)+(1-\lambda)f_1(x_1)\le \lambda f(x_0)+(1-\lambda)f(x_1),
\]
from the definition of $f$.

Next, assume  $\bar x\in\cup_{j\in J} I_j$. Then there is a unique index $j_0\in J$ such that $\bar x\in I_{j_0}$ with $y_0$ and $y_1$ denoting the left- and right-end points of $I_{j_0}$, respectively. If $x_0\ge y_0$ and $x_1\le y_1$, then
\[
f(\bar x)=f_2(\bar x)\le \lambda f_2(x_0)+(1-\lambda) f_2(x_1)=\lambda f(x_0)+(1-\lambda) f(x_1).
\]
Suppose $x_0<y_0$ and $x_1\le y_1$. As $\bar x\in I_{j_0}$, we have $\bar{x}=\mu y_0+(1-\mu) x_1$ and $y_0=\nu x_0+(1-\nu) x_1$ for some $0<\mu,\nu<1$. So
\[
\lambda x_0+(1-\lambda) x_1=\bar x=\mu\nu x_0+(1-\mu\nu)x_1;
\]
hence $\mu\nu=\lambda$, and we get
\[
\begin{split}
f(\bar x) & =f_2(\bar x)=f_2(\mu y_0+(1-\mu)x_1)\le \mu f_2(y_0)+(1-\mu) f_2(x_1) \\
& =\mu f_1(y_0)+(1-\mu) f_2(x_1)=\mu f_1(\nu x_0+(1-\nu)x_1)+(1-\mu) f_2(x_1) \\
& \le\mu\nu f_1(x_0)+\mu(1-\nu)f_1(x_1)+(1-\mu) f_2(x_1) \\
& \le \lambda f_1(x_0)+(1-\lambda) f_2(x_1)\le \lambda f(x_0)+(1-\lambda) f(x_1).
\end{split}
\]
The other cases can be handled similarly; we omit these.

In any case, inequality (\ref{convex-inqeuality}) holds, and the proof is complete.
\end{proof}

We now finish the proof of Lemma \ref{lem:rank-1-property}.

\begin{proof}[Proof of Lemma \ref{lem:rank-1-property}]
Let $[\xi_1,\xi_2]$ be any rank-one  segment in $\Sigma_t$; that is, $\xi_1,\xi_2\in\Sigma_t$ and $\mathrm{rank}(\xi_1-\xi_2)=1$. It suffices to show that the function $s\mapsto f(s\xi_1+(1-s)\xi_2)$ $(s\in[0,1])$ is convex.

If $[\xi_1,\xi_2]\subset V$ or $[\xi_1,\xi_2]\subset\Sigma_t\setminus V$, there is nothing to show. So we assume both inclusions do not hold; that is,
\[
[\xi_1,\xi_2]\cap V\not=\emptyset\quad\mbox{and}\quad [\xi_1,\xi_2]\cap (\Sigma_t\setminus V)\not=\emptyset.
\]
Consider the 1-1 map $s\mapsto s\xi_1+(1-s)\xi_2=:H(s)$ from $[0,1]$ onto $[\xi_1,\xi_2]$. Let $\{I_j\}_{j\in J}$ be the countable collection of disjoint open intervals in $(0,1)$ such that $H(\cup_{j\in J}I_j)=(\xi_1,\xi_2)\cap V$. We can now apply Lemma \ref{lem:rank-1-property-1D} to the function $f\circ H:[0,1]\to\R$ to conclude that it is convex. 

One can repeat the same proof for the unconstrained case by replacing $\Sigma_t$ with $\M^{m\times n}$.
\end{proof}

The following lemma is on the identification of the rank-one convex hull of a compact set as the zero level set of some nonnegative rank-one convex function.

\begin{lem}\label{lem:zero-level-set}
Let $K$ be a compact subset of $\Sigma_t$, and let $\tilde K=K^{rc,\Sigma_t}.$ Then there exists a nonnegative rank-one convex function $g:\Sigma_t\to\R$ such that
\[
\tilde K=\{\xi\in\Sigma_t\,|\, g(\xi)=0\}.
\]
The same result remains true when $\Sigma_t$ and $K^{rc,\Sigma_t}$ are replaced by $\M^{m\times n}$ and $K^{rc},$ respectively.
\end{lem}

\begin{proof}
For each $r>0$, set $\Sigma_{t,r}=\{\xi\in\Sigma_t\,|\, |\xi|<r\}.$ Choose an $R>0$ so large that $\tilde K\subset\Sigma_{t,R/2}$. Define a function $g_1:\Sigma_{t,R}\to\R$ by
\[
g_1(\xi)=\sup\big\{f(\xi)\,|\, \mbox{$f:\Sigma_{t,R}\to\R$ is rank-one convex, $f\le\mathrm{dist}(\cdot,K)$ in $\Sigma_{t,R}$}\big\}
\]
for all $\xi\in\Sigma_{t,R}$.
As the zero function $0\le\mathrm{dist}(\cdot,K)$ in $\Sigma_{t,R}$ is rank-one convex, we have $g_1\ge 0$ in $\Sigma_{t,R}$. It is also easy to see that $g_1:\Sigma_{t,R}\to\R$ is rank-one convex and that $K\subset\{\xi\in\Sigma_{t,R}\,|\,g_1(\xi)=0\}$. 

Let us check that $g_1>0$ in $\Sigma_{t,R}\setminus\tilde K$. To see this, let $\xi\in\Sigma_{t,R}\setminus\tilde K$. By definition, $\alpha:=\tilde f(\xi)>0$ for some rank-one convex function $\tilde f:\Sigma_{t}\to\R$ with $\tilde f\le 0$ on $K$. We now verify that
\begin{equation}\label{claim-M-finite}
0<M:=\sup_{\eta\in\Sigma_{t,R}\setminus K}\frac{\tilde f(\eta)-\frac{\alpha}{2}}{\mathrm{dist}(\eta,K)}<\infty.
\end{equation}
Note first that
\[
0<\frac{\alpha}{2\cdot\mathrm{dist}(\xi,K)}=\frac{\tilde f(\xi)-\frac{\alpha}{2}}{\mathrm{dist}(\xi,K)}\le M.
\]
It thus remains to show that $M<\infty.$ Suppose on the contrary that $M=\infty$. Then we can choose a sequence $\{\eta_j\}_{j\in\N}$ in $\Sigma_{t,R}\setminus K$ so that
\begin{equation}\label{seq-to-infty}
\frac{\tilde f(\eta_j)-\frac{\alpha}{2}}{\mathrm{dist}(\eta_j,K)}\to\infty \quad\mbox{as $j\to\infty$}.
\end{equation}
Passing to a subsequence if necessary, we can assume $\eta_j\to\eta_0$ for some $\eta_0\in\bar\Sigma_{t,R}.$ If $\eta_0\not\in K$, then
\begin{equation*}
\frac{\tilde f(\eta_j)-\frac{\alpha}{2}}{\mathrm{dist}(\eta_j,K)}\to \frac{\tilde f(\eta_0)-\frac{\alpha}{2}}{\mathrm{dist}(\eta_0,K)}\in\R \quad\mbox{as $j\to\infty$},
\end{equation*}
which is a contradiction to (\ref{seq-to-infty}). If $\eta_0\in K,$ then $\tilde f(\eta_0)\le 0$, and so
\begin{equation*}
\frac{\tilde f(\eta_j)-\frac{\alpha}{2}}{\mathrm{dist}(\eta_j,K)}\to -\infty \quad\mbox{as $j\to\infty$},
\end{equation*}
which is also a contradiction to (\ref{seq-to-infty}). Thus (\ref{claim-M-finite}) holds, and this implies that
\[
\frac{1}{M}(\tilde f(\eta)-\frac{\alpha}{2})\le \mathrm{dist}(\eta,K)\quad \forall \eta\in\Sigma_{t,R}.
\]
As $\frac{1}{M}(\tilde f-\frac{\alpha}{2}):\Sigma_{t,R}\to\R$ is rank-one convex, it now follows from the definition of $g_1$ that $\frac{1}{M}(\tilde f(\eta)-\frac{\alpha}{2})\le g_1(\eta)$ for all $\eta\in\Sigma_{t,R}$. In particular,
\[
0<\frac{\alpha}{2M}=\frac{1}{M}(\tilde f(\xi)-\frac{\alpha}{2})\le g_1(\xi).
\]
Therefore, $g_1>0$ in $\Sigma_{t,R}\setminus\tilde K$.

Let $g:\Sigma_t\to\R$ be the function defined by
\[
g(\xi)=\left\{\begin{array}{ll}
         \max\{g_1(\xi),12|\xi|-9R\}, & \xi\in\Sigma_{t,R}, \\
         12|\xi|-9R, & \xi\in\Sigma_t,\,|\xi|\ge R.
       \end{array}\right.
\]
Lastly, we check that $g$ is the desired function. Clearly, $g\ge 0$ in $\Sigma_t$. Since $g_1(\xi)\le\mathrm{dist}(\xi,K)<\frac{3R}{2}$ for all $\xi\in\Sigma_{t,R}$, we have $g(\xi)=12|\xi|-9R$ for all $\xi$ in some neighborhood of $\{\eta\in\Sigma_{t}\,|\,|\eta|=R\}$ in $\Sigma_{t}$. Thus it follows from Lemma \ref{lem:rank-1-property} that $g:\Sigma_{t}\to\R$ is rank-one convex. Next, we verify that
\[
\tilde K=\{\xi\in\Sigma_t\,|\,g(\xi)=0\}.
\]
As $\tilde K\subset\Sigma_{t,R/2}$, it follows from the above fact that $g_1(\xi)>0$ for all $\xi\in\Sigma_t$ with $R/2\le|\xi|<R$, and so $g(\xi)>0$ for all $\xi\in\Sigma_t$ with $|\xi|\ge R/2$. This implies that
\[
\{\xi\in\Sigma_t\,|\, g(\xi)=0\}=\{\xi\in\Sigma_t\,|\,|\xi|<R/2,\, g_1(\xi)=0\}\subset\tilde K.
\]
To show the reverse inclusion, let $\xi\in\tilde K$. Let $f:\Sigma_{t,R}\to\R$ be any rank-one convex function such that $f\le\mathrm{dist}(\cdot,K)$ in $\Sigma_{t,R}$. If we can show that $f(\xi)\le 0$, then the definition of $g_1$ implies that $g(\xi)=g_1(\xi)=0$, and the proof is complete for the case of the linear constraint. Suppose on the contrary that $f(\xi)>0$.
Define
\[
\tilde f(\eta)=\left\{\begin{array}{ll}
         \max\{f(\eta),12|\eta|-9R\}, & \eta\in\Sigma_{t,R}, \\
         12|\eta|-9R, & \eta\in\Sigma_t,\,|\eta|\ge R.
       \end{array}\right.
\]
Then $\tilde f:\Sigma_t\to\R$ is rank-one convex as above. Also, $\tilde f(\xi)= f(\xi)>0$. As $\tilde f\le 0$ on $K$, we now have $\xi\not\in \tilde K$; a contradiction. Thus $f(\xi)\le 0$.

For the unconstrained case, one can repeat the same proof with $\Sigma_t$, $\Sigma_{t,r}$ $(r>0)$ and $K^{rc,\Sigma_t}$ replaced by $\M^{m\times n}$, $B_r=\{\xi\in\M^{m\times n}\,|\,|\xi|<r\}$ $(r>0)$ and $K^{rc}$, respectively.
\end{proof}

Using the previous lemma, we obtain the following.

\begin{lem}\label{lem:rank-1-extension}
Let $K$ be a compact subset of $\Sigma_t$, let $\OO$ be an open set in $\Sigma_t$ containing $\tilde K:=K^{rc,\Sigma_t}$, and let $f:\OO\to\R$ be a rank-one convex function. Then there exists a rank-one convex function $F:\Sigma_t\to\R$ such that
\[
F\equiv f\quad\mbox{in some neighborhood of $\tilde K$ in $\OO$}.
\]
The same result holds when $\Sigma_t$ and $K^{rc,\Sigma_t}$ are replaced by $\M^{m\times n}$ and $K^{rc}$, respectively.
\end{lem}

\begin{proof}
We first use Lemma \ref{lem:zero-level-set} to obtain a nonnegative rank-one convex function $g:\Sigma_t\to\R$ such that $\tilde K=\{\xi\in\Sigma_t\,|\,g(\xi)=0\}$. Set $m=\min_{\tilde K}f$. Choose a number $c>0$ so that $m+c>0.$ Define $\tilde f=f+c$ in $\OO$; then $\min_{\tilde K}\tilde f=m+c>0$. For each $\xi\in \tilde K$, we thus can choose an open ball $B_\xi$ in $\Sigma_t$ with $\bar B_\xi\subset\OO$ and center $\xi$ such that $\tilde f>0$ on $\bar B_\xi$. As $\tilde K$ is compact, we can choose finitely many matrices $\xi_1,\cdots,\xi_N\in \tilde{K}$ such that $\tilde{K}\subset\cup_{j=1}^N B_{\xi_j}=:U$; so $\tilde f>0$ on $\bar U\subset\OO$.

For each $k\in\N$, let
\[
U_k=\{\xi\in\OO\,|\,\tilde f(\xi)>k g(\xi)\},
\]
which is open in $\OO$, and let $V_k$ be the union of all connected components of $U_k$ that have a nonempty intersection with $\tilde K$; then $\tilde{K}\subset V_k\subset U_k$. For each $\delta>0$, let $S_\delta=\{\xi\in\Sigma_t\,|\,\mathrm{dist}(\xi,\tilde{K})<\delta\}$. Fix a $\delta>0$ so small that $\bar S_\delta\subset U.$
As $g>0$ on $\overline{U\setminus S_\delta}$, we can choose a number $k\in\N$ so large that
\[
k\cdot \min_{\overline{U\setminus S_\delta}} g\ge\max_{\overline{U\setminus S_\delta}}\tilde f,
\]
and so $U_{k}\cap (\overline{U\setminus S_\delta})=\emptyset.$ Thus we easily see from the definition of $V_k$ that $V_{k}\subset  S_{\delta}$.

Next, define
\[
\tilde F(\xi)=\left\{\begin{array}{ll}
                \tilde f(\xi), & \xi\in V_{k}, \\
                k g(\xi), & \xi\in\Sigma_t\setminus V_{k};
              \end{array}\right.
\]
then, by Lemma \ref{lem:rank-1-property}, $\tilde F:\Sigma_t\to\R$ is rank-one convex. Take $F=\tilde F-c$ in $\Sigma_t;$ then $F:\Sigma_t\to\R$ is rank-one convex and $F\equiv f$ in $V_k$, where $\tilde K\subset V_k\subset\OO.$

For the unconstrained case, one can repeat the same proof with $\Sigma_t$ and $K^{rc,\Sigma_t}$ replaced by $\M^{m\times n}$ and $K^{rc}$, respectively.
\end{proof}

Note that for each $\xi\in\M^{m\times n}$, there exists a unique number $s_{\xi}\in\R$ such that $\pi(\xi):=\xi+s_{\xi}L/|L|\in\Sigma_t$; so $\xi=\pi(\xi)+t_{\xi}L/|L|$, where $t_\xi:=-s_\xi$.
As the last preparation for the proof of Lemma \ref{lem:main-approximation}, we prove the following lemma.

\begin{lem}\label{lem:approx-linear-whole}
Let $f:\Sigma_t\to\R$ be a smooth rank-one convex function.
For each $\epsilon>0$ and each $k>0$, let $F_{\epsilon,k}:\M^{m\times n}\to\R$ be the function defined by
\[
F_{\epsilon,k}(\xi)=f(\pi(\xi))+\epsilon|\xi|^2+k|\mathcal{L}(\xi)-t|^2\quad\forall \xi\in\M^{m\times n}.
\]
Let $K$ be a compact subset of $\Sigma_t$. Then for each $\epsilon>0$, there exists a number $k>0$ such that $F_{\epsilon,k}:U_{\epsilon,k}\to\R$ is rank-one convex, for some open set $U_{\epsilon,k}$ in $\M^{m\times n}$ containing $K$.
\end{lem}

\begin{proof}
We prove by contradiction; suppose there exists an $\epsilon>0$ such that for each $k>0$, if $V$ is any open set in $\M^{m\times n}$ containing $K$, then $F_{\epsilon,k}:V\to\R$ is not rank-one convex.

Let $k\in\N$, and set
\[
V_k=\{\xi\in\M^{m\times n}\,|\,\mathrm{dist}(\xi,K)<1/k\}.
\]
Then $F_{\epsilon,k}:V_k\to\R$ is not rank-one convex. By the Legendre-Hadamard condition, there exist a matrix $\eta^k\in V_k$ and vectors $\lambda^k\in\R^m$, $\mu^k\in\R^n$ with $|\lambda^k|=|\mu^k|=1$ such that
\begin{equation}\label{condition-LH}
D^2F_{\epsilon,k}(\eta^k)(\lambda^k\otimes\mu^k,\lambda^k\otimes\mu^k)=\sum_{1\le i,j\le m,\,1\le \alpha,\beta\le n} \frac{\partial^2 F_{\epsilon,k}(\xi)}{\partial\xi_{i\alpha}\partial\xi_{j\beta}}\Bigg|_{\xi=\eta^k} \lambda^k_i\lambda^k_j\mu^k_\alpha\mu^k_\beta < 0.
\end{equation}
Passing to a subsequence if necessary, we have
\[
\mbox{$\eta^k\to \eta$\;\; in $\M^{m\times n}$, $\lambda^k\to\lambda$\;\; in $\R^m$ and $\mu^k\to\mu$\;\; in $\R^n$,}
\]
for some $\eta\in K$, $\lambda\in\R^m$ and $\mu\in\R^n$ with $|\lambda|=|\mu|=1$.

Set $g(\xi)=|\mathcal{L}(\xi)-t|^2$ for all $\xi\in\M^{m\times n}$. Note here that $\mathcal{L}(\eta^k+s \lambda^k\otimes\mu^k)=a_k s+b_k$ for all $s\in\R$, where $a_k=L\cdot (\lambda^k\otimes \mu^k)$ and $b_k=L\cdot\eta^k$. So
\[
g(\eta^k+s\lambda^k\otimes\mu^k)=|a_k s+b_k-t|^2=a_k^2s^2+2a_k(b_k-t)s+(b_k-t)^2.
\]
From this, we get
\begin{equation}\label{2nd-derivative-g}
\frac{d^2}{ds^2}g(\eta^k+s\lambda^k\otimes\mu^k)=2a_k^2.
\end{equation}
Differentiating $F_{\epsilon,k}(\eta^k+s\lambda^k\otimes\mu^k)$ twice with respect to the variable $s$ and then letting $s=0$, we obtain from (\ref{condition-LH}) and (\ref{2nd-derivative-g}) that
\[
D^2(f\circ \pi)(\eta^k)(\lambda^k\otimes\mu^k,\lambda^k\otimes\mu^k)+2\epsilon+2k a_k^2 <0.
\]
Taking the limit supremum as $k\to\infty$, we get
\[
D^2(f\circ \pi)(\eta)(\lambda\otimes\mu,\lambda\otimes\mu)+2\epsilon+2\cdot\limsup_{k\to\infty}k a_k^2 \le0;
\]
thus $\lim_{k\to\infty}a_k=0$ and
$
D^2(f\circ \pi)(\eta)(\lambda\otimes\mu,\lambda\otimes\mu)+2\epsilon\le 0.
$

Let $l>0$. Then
\[
D^2 F_{\epsilon,l}(\eta^k)(\lambda^k\otimes\mu^k,\lambda^k\otimes\mu^k)= D^2(f\circ \pi)(\eta^k)(\lambda^k\otimes\mu^k,\lambda^k\otimes\mu^k)
 +2\epsilon+2l a_k^2.
\]
Letting $k\to\infty$, we get
\[
D^2 F_{\epsilon,l}(\eta)(\lambda\otimes\mu,\lambda\otimes\mu)=D^2(f\circ \pi)(\eta)(\lambda\otimes\mu,\lambda\otimes\mu)+2\epsilon \le 0;
\]
that is,
\begin{equation}\label{contradiction-inequality}
D^2 F_{\epsilon,l}(\eta)(\lambda\otimes\mu,\lambda\otimes\mu)\le 0\quad\forall l>0.
\end{equation}

Next, observe
\[
\frac{d}{ds}\mathcal{L}(\eta^k+s \lambda^k\otimes \mu^k)=a_k;
\]
thus letting $k\to\infty$, we get
\[
\frac{d}{ds}\mathcal{L}(\eta+s \lambda\otimes \mu)=0\quad\forall s\in\R.
\]
Thus $\mathcal{L}(\eta+s\lambda\otimes\mu)=\mathcal{L}(\eta)=t$ for all $s\in\R$.
Let $l>0$. We now have
\[
F_{\epsilon,l}(\eta+s\lambda\otimes\mu)=f(\eta+s\lambda\otimes\mu)+\epsilon|\eta+s\lambda\otimes \mu|^2.
\]
Since the function $s\mapsto f(\eta+s\lambda\otimes\mu)$ $(s\in\R)$ is convex, we have
\[
\frac{d^2}{ds^2}f(\eta+s\lambda\otimes\mu)\ge 0\quad \forall s\in\R.
\]
Thus,
\[
D^2 F_{\epsilon,l}(\eta)(\lambda\otimes\mu,\lambda\otimes\mu)=\frac{d^2}{ds^2} F_{\epsilon,l}(\eta+s\lambda\otimes\mu)\Big|_{s=0}\ge 2\epsilon\quad\forall l>0;
\]
this is a contradiction to (\ref{contradiction-inequality}), and the proof is complete.
\end{proof}

We are now ready to prove Lemma \ref{lem:main-approximation}.

\begin{proof}[Proof of Lemma \ref{lem:main-approximation}]
Using Lemma \ref{lem:rank-1-extension}, we can find a rank-one convex function $g:\Sigma_t\to\R$ such that $g=f$ on $\tilde K$. Then we choose an open ball $B$ in $\M^{m\times n}$ containing $\tilde K$ and set $R=\sup_{\xi\in B\cap\Sigma_t}|\xi|$.

Let $\epsilon>0$. Upon on mollifying the function $g$, we can find a smooth rank-one convex function $\tilde g:\Sigma_t\to\R$ such that $|\tilde g-g|<\epsilon/2$ on the compact set $\bar B\cap\Sigma_t$.

For each $k>0$, let $\tilde G_{\epsilon,k}:\M^{m\times n}\to\R$ be the function defined by
\[
\tilde G_{\epsilon,k}(\xi)=\tilde g(\pi(\xi))+\frac{\epsilon}{2R^2}|\xi|^2+k|\mathcal{L}(\xi)-t|^2\quad \forall\xi\in\M^{m\times n}.
\]
Then by Lemma \ref{lem:approx-linear-whole}, there exists a number $k>0$ such that $\tilde G_{\epsilon,k}:U_{\epsilon,k}\to\R$ is rank-one convex, for some open set $U_{\epsilon,k}$ in $\M^{m\times n}$ containing $\bar B\cap\Sigma_t$. Let us write $G=\tilde G_{\epsilon,k}:U_{\epsilon,k}\to\R$. Note that for all $\xi\in\tilde K\subset B\cap\Sigma_t$,
\[
|G(\xi)-g(\xi)|\le |G(\xi)-\tilde g(\xi)|+|\tilde g(\xi)-g(\xi)|<\frac{\epsilon}{2R^2}R^2+\frac{\epsilon}{2}=\epsilon.
\]
Observe $(\bar{B}\cap\Sigma_t)^{rc}=\bar{B}\cap\Sigma_t\subset U_{\epsilon,k}$. Applying Lemma \ref{lem:rank-1-extension}, we can choose a rank-one convex function $F:\M^{m\times n}\to \R$ such that $F=G$ on $\bar{B}\cap\Sigma_t$. Thus
\[
|F-f|=|G-g|<\epsilon\quad\mbox{on $\tilde K$}.
\]
\end{proof}


We finally get to the proof of Theorem \ref{thm:main-lemma}.

\begin{proof}[Proof of Theorem \ref{thm:main-lemma}]
Recall from Proposition \ref{prop:rank-one-1} and \cite{Pe1} that
\[
K^{rc,\Sigma_t}\subset K^{rc}=\{\bar\mu\,|\,\mu\in\mathcal{M}^{rc}(K)\}\subset\Sigma_t.
\]
Let $\nu\in\mathcal{M}^{rc}(K).$ To show the reverse inclusion $K^{rc,\Sigma_t}\supset K^{rc}$, it suffices to check that $\bar\nu\in \tilde K:= K^{rc,\Sigma_t}$. To prove by contradiction, suppose $\bar\nu\in\Sigma_t\setminus\tilde K$. Then there exists a rank-one convex function $g:\Sigma_t\to\R$ with $g\le 0$ on $K$ such that $g(\bar\nu)>0$; so $\langle\nu,g\rangle\le 0<g(\bar\nu)$. Then by Lemma \ref{lem:main-approximation}, for any given $\epsilon>0$ to be specified below, we get a rank-one convex function $h:\M^{m\times n}\to\R$ such that $|h-g|<\epsilon$ on the compact set $K\cup\{\bar\nu\}$. This implies that $\langle\nu,h\rangle<\langle\nu,g\rangle+\epsilon<g(\bar\nu)-\epsilon<h(\bar\nu)$ if $\epsilon>0$ is chosen so small that $\epsilon<\frac{g(\bar\nu)-\langle\nu,g\rangle}{2}$. In short, we have $\langle\nu,h\rangle<h(\bar\nu)$; a contradiction to the fact that $\nu$ is a laminate. Thus $\bar\nu\in\tilde{K}$, and so $K^{rc}=K^{rc,\Sigma_t}$.

Next, let $\OO$ be an open set in $\Sigma_t$ containing $K^{rc}$.
We choose a bounded open set $U$ in $\Sigma_t$ such that $K^{rc}\subset U\subset\bar{U}\subset\OO.$ Set $\mathcal{F}=\{\mu\in\mathfrak{L}(U)\,|\,\bar\mu=\bar\nu\}$; then $\delta_{\bar\nu}\in\mathcal{F}\neq\emptyset$. To finish the proof, it suffices to show that the weak* closure $\bar{\mathcal{F}}^*$ of $\mathcal{F}$ in $\mathcal{P}$ contains $\nu$. We prove by contradiction; so suppose $\nu\not\in\bar{\mathcal{F}}^*$. Since $\mathcal{F}$ is convex, it follows from the Hahn-Banach Theorem that there exists a continuous function $f:\bar U\to\R$ such that
\[
\langle \nu,f\rangle<\inf\{\langle\mu,f\rangle\,|\,\mu\in\mathcal{F}\}.
\]
Since $\bar U$ is compact, it follows from Lemma \ref{lem:rank-one-envelope} that
\[
R_{U}f(\bar\nu)=\inf\{\langle\mu,f\rangle\,|\,\mu\in\mathcal{F}\}.
\]
Note that $R_U f:U\to\R$ is a rank-one convex function with $R_U f\le f$ in $U$. From the above observation, we have $\langle\nu,R_U f\rangle\le\langle \nu,f\rangle< R_U f(\bar\nu)$. By Lemma \ref{lem:main-approximation}, for any given $\epsilon>0$ to be chosen below, we obtain a rank-one convex function $F:\M^{m\times n}\to\R$ such that $|F-R_U f|<\epsilon$ on $\tilde K=K^{rc}$. Since $\bar\nu\in K^{rc}$, we choose $0<\epsilon<\frac{R_U f(\bar\nu)-\langle\nu,R_U f\rangle}{2}$ to have $\langle \nu,F\rangle< F(\bar\nu)$; a contradiction to the fact that $\nu$ is a laminate.

The proof is now complete.
\end{proof}


\section{Rank-one smooth approximation under  linear constraint}\label{sec:rank-1}
We begin this section by introducing a pivotal approximation result, Theorem \ref{thm:rank-1}, for proving the main results of the paper, Theorems \ref{thm:main-1} and \ref{thm:main-2}. Its special cases have been successfully applied to some nonstandard evolution problems \cite{KY1,KY2,KY3,KK1}. Although the proof of Theorem \ref{thm:rank-1} already appeared in \cite{KK1}, we include it in Section \ref{sec:proof-rank-1} for the sake of completeness as we make use of the general version of the theorem for the first time in this paper.

\begin{thm}\label{thm:rank-1}
Let $m,n\ge 2$ be integers, and let $A,B\in\M^{m\times n}$ be such that $\rank(A-B)=1$; hence
\[
A-B=a\otimes b=(a_i b_j)
\]
for some nonzero vectors $a\in\R^m$ and $b\in\R^n$ with $|b|=1.$
Let $L\in\M^{m\times n}$ satisfy
\begin{equation}\label{rank-1-1}
Lb\ne 0 \;\;\mbox{in}\;\;\R^m,
\end{equation}
and let $\mathcal{L}:\M^{m\times n}\to \R$ be the linear function defined by
\[
\mathcal{L}(\xi)=\sum_{1\le i\le m,\, 1\le j \le n} L_{ij}\xi_{ij}\quad \forall \xi\in\M^{m\times n}.
\]
Assume $\mathcal{L}(A)=\mathcal{L}(B)$ and $0<\lambda<1$ is any fixed number. Then there exists a linear partial differential operator $\Phi:C^1(\R^n;\R^m)\to C^0(\R^n;\R^m)$ satisfying the following properties:

(1) For any open set $\Omega\subset\R^n$,
\[
\Phi v\in C^{k-1}(\Omega;\R^m)\;\;\mbox{whenever}\;\; k\in\N\;\;\mbox{and}\;\;v\in C^{k}(\Omega;\R^m)
\]
and
\[
\mathcal{L}(\nabla\Phi v)=0 \;\;\mbox{in}\;\;\Omega\;\;\forall v\in C^2(\Omega;\R^m).
\]

(2) Let $\Omega\subset\R^n$ be any bounded domain. For each $\tau>0$, there exist a function $g=g_\tau\in  C^{\infty}_{c}(\Omega;\R^m)$ and two disjoint open sets $\Omega_A,\Omega_B\subset\subset\Omega$ such that
\begin{itemize}
\item[(a)] $\Phi g\in C^\infty_c(\Omega;\R^m)$,
\item[(b)] $\dist(\nabla\Phi g,[-\lambda(A-B),(1-\lambda)(A-B)])<\tau$ in $\Omega$,
\item[(c)] $\nabla \Phi g(x)= \left\{\begin{array}{ll}
                                (1-\lambda)(A-B) & \mbox{$\forall x\in\Omega_A$}, \\
                                -\lambda(A-B) & \mbox{$\forall x\in\Omega_B$},
                              \end{array}\right.$
\item[(d)] $\big||\Omega_A|-\lambda|\Omega|\big|<\tau$, $\big||\Omega_B|-(1-\lambda)|\Omega|\big|<\tau$,
\item[(e)] $\|\Phi g\|_{L^\infty(\Omega)}<\tau$,
\end{itemize}
where $[-\lambda(A-B),(1-\lambda)(A-B)]$ is the closed line segment in $\mathrm{ker}\mathcal{L}\subset\M^{m\times n}$ joining $-\lambda(A-B)$ and $(1-\lambda)(A-B)$.
\end{thm}

Using this theorem, we deduce a preliminary result towards Theorems \ref{thm:main-1} and \ref{thm:main-2}. We remark that Lemma \ref{lem:key-lemma} is the spot where the two major tools of the paper, Theorems \ref{thm:main-lemma} and \ref{thm:rank-1}, meet. Note also that piecewise linear approximation scheme is not used here and below in Lemma \ref{lem:pre-result} (cf. \cite[Lemma 4.1]{MSv1}).

\begin{lem}\label{lem:key-lemma}
Assume that (\ref{assume-1}) is satisfied. Let $\Omega\subset\R^n$ be a bounded domain, let $V$ be an open set in $\Sigma_t$, and let $\xi\in V^{rc}$. Then for each $\epsilon>0$, there exists a map $\varphi\in C^{\infty}_c(\Omega;\R^m)$ such that
\[
\left\{\begin{array}{l}
         \xi+\nabla \varphi\in V^{rc}\quad\mbox{in $\Omega$}, \\
         \big|\{x\in\Omega\,|\,\xi+\nabla \varphi(x)\not\in V\}\big|<\epsilon|\Omega|, \\
         \|\varphi\|_{L^\infty(\Omega)}<\epsilon.
       \end{array}
\right.
\]
\end{lem}

\begin{proof}
As $\xi\in V^{rc}$, there exists a compact set $K\subset V$ such that $\xi\in K^{rc}=\{\bar\mu\,|\,\mu\in\mathcal{M}^{rc}(K)\}$. So $\xi=\bar\nu$ for some $\nu\in\mathcal{M}^{rc}(K)$. From Corollary  \ref{prop:rank-1-hull-open}, we see that $V^{rc}$ is  open in $\Sigma_t$. We thus can apply Theorem \ref{thm:main-lemma} to extract a sequence $\nu_k\in\mathfrak{L}(V^{rc})$ with $\bar{\nu}_k=\bar\nu=\xi$ that converges weakly* to $\nu$ in $\mathcal{P}$.

\textbf{\underline{Claim:}} For each $\mu\in\mathfrak{L}(V^{rc})$ of order $N-1\ge 0$ with $\mu=\sum_{j=1}^N\lambda_j\delta_{\xi_j}$, there exists an $\eta_\mu>0$ such that for each  $0<\eta<\eta_\mu$ and each $\epsilon>0$, there is a map $\varphi\in C^\infty_c(\Omega;\R^m)$ satisfying
\begin{equation}\label{pre-rank-1-approximation}
\left\{\begin{array}{l}
         \bar\mu+\nabla \varphi\in V^{rc}\quad\mbox{in $\Omega$}, \\
         \Big|\big|\{x\in\Omega\,|\,|
         \bar\mu+\nabla \varphi(x)-\xi_j|<\eta\}\big|-\lambda_j|\Omega|\Big|<\epsilon|\Omega|\quad\mbox{for all $1\le j\le N$},\\
         \|\varphi\|_{L^\infty(\Omega)}<\epsilon.
       \end{array}
\right.
\end{equation}

Suppose for the moment that Claim holds. Choose a function $F\in C^\infty_c(\tilde V)$ so that $0\le F\le 1$ in $\tilde V$ and $F\equiv 1$ on $K$, where $\tilde{V}$ is some open set in $\M^{m\times n}$ with $V=\tilde{V}\cap\Sigma_t$. Let $0<\epsilon\le 2$. Since
\[
\int_{\M^{m\times n}}F d\nu_k \to \int_{\M^{m\times n}}F d\nu=1\quad\mbox{as $k\to\infty$},
\]
we can choose an index $i\in\N$ so large that
\begin{equation}\label{weakstar-consequence}
0\le 1-\sum_{j=1}^N\lambda_j F(\xi_j)=1-\int_{\M^{m\times n}}F d\nu_i<\frac{\epsilon}{2},
\end{equation}
where $\nu_i=\sum_{j=1}^N\lambda_j\delta_{\xi_j}\in\mathfrak{L}(V^{rc})$ is of order $N-1\ge 0$. Set $J_V=\{j\in\{1,\cdots,N\}\,|\,\xi_j\in V\}.$
If $J_V=\emptyset$, then $1=1-\sum_{j=1}^N\lambda_j F(\xi_j)<\frac{\epsilon}{2}$, a contradiction. Thus $J_V\neq \emptyset$. Now, let
\begin{equation}\label{choice-eta}
0<\eta<\min\Big\{\min_{j\in J_V}\mathrm{dist}(\xi_j,\partial|_{\Sigma_t}V),\min_{j,k\in J_V,\,j\neq k}2^{-1}|\xi_j-\xi_k|,\eta_{\nu_i}\Big\},
\end{equation}
where the number $\eta_{\nu_i}>0$ is from the result of Claim above. It then follows from the result of Claim that there exists a map $\varphi\in C^\infty_c(\Omega;\R^m)$ such that
\begin{equation}\label{phi-function-result}
\left\{\begin{array}{l}
         \xi+\nabla \varphi=\bar\nu_i+\nabla \varphi\in V^{rc}\quad\mbox{in $\Omega$}, \\
         \big||\{x\in\Omega\,|\,|
         \xi+\nabla \varphi(x)-\xi_j|<\eta\}|-\lambda_j|\Omega|\big|<\frac{\epsilon}{2N}|\Omega|\quad\mbox{for all $1\le j\le N$},\\
         \|\varphi\|_{L^\infty(\Omega)}<\epsilon.
       \end{array}
\right.
\end{equation}
Thus, by (\ref{weakstar-consequence}), (\ref{choice-eta}) and (\ref{phi-function-result}), we have
\[
\begin{split}
|\{x\in & \Omega\,|\,\xi+\nabla \varphi(x)\not\in V\}|=|\Omega|- |\{x\in\Omega\,|\,\xi+\nabla \varphi(x)\in V\}| \\
\le & |\Omega|-|\Omega|\sum_{j\in J_V}\lambda_j F(\xi_j)+\sum_{j\in J_V}\lambda_j |\Omega| \\
& -\sum_{j\in J_V}\big|\{x\in\Omega\,|\,|
         \xi+\nabla \varphi(x)-\xi_j|<\eta\}\big|<\frac{\epsilon|\Omega|}{2}+\frac{\epsilon|\Omega|}{2} =\epsilon|\Omega|;
\end{split}
\]
hence the map $\varphi$ satisfies the required properties for the conclusion of the lemma.

It now remains to prove Claim above. Let us prove this by induction on the order $l\ge 0$ of a laminate $\mu=\sum_{j=1}^{l+1}\lambda_j\delta_{\xi_j}\in\mathfrak{L}(V^{rc})$.
If the order $l=0$, we can simply take $\varphi\equiv0$ in $\Omega$; then (\ref{pre-rank-1-approximation}) holds for all $\eta>0$ and $\epsilon>0$.

Next, assume  that the assertion holds for the order $l= k$, where $k\ge 0$ is an integer. Let $\mu=\sum_{j=1}^{k+2}\lambda_j\delta_{\xi_j}\in\mathfrak{L}(V^{rc})$ be of a laminate of order $l=k+1$. Reordering the indices $j$ in $\mu$ if necessary and setting
\[
\tilde\lambda_{k+1}=\lambda_{k+1}+\lambda_{k+2},\quad\tilde\lambda=\frac{\lambda_{k+1}}{\tilde{\lambda}_{k+1}}\in(0,1),
\]
\[
\tilde\xi_{k+1}=\tilde\lambda\xi_{k+1}+ (1-\tilde\lambda)\xi_{k+2},\quad\tilde\mu=\sum_{j=1}^k\lambda_j\delta_{\xi_j}+\tilde\lambda_{k+1}\delta_{\tilde\xi_{k+1}},
\]
it follows that $[\xi_{k+1},\xi_{k+2}]$ is a rank-one segment in $V^{rc}$ and that
\[
\mu=\tilde\mu-\tilde\lambda_{k+1}\delta_{\tilde\xi_{k+1}} +\tilde\lambda\tilde\lambda_{k+1}\delta_{\xi_{k+1}} +(1-\tilde\lambda)\tilde\lambda_{k+1}\delta_{\xi_{k+2}},
\]
where $\tilde\mu$ is a laminate of order $k$ in $V^{rc}$.
Let $\epsilon>0$ and
\[
\begin{split}
0<\eta<\frac{1}{2}\min\Big\{ & \min_{1\le i,j\le k+2,\, i\ne j}|\xi_i-\xi_j|,\min_{1\le j\le k}|\xi_j-\tilde\xi_{k+1}|,\\
& \mathrm{dist}([\xi_{k+1},\xi_{k+2}],\partial|_{\Sigma_t}V^{rc}),\eta_{\tilde\mu}\Big\}=:\eta_\mu,
\end{split}
\]
where the number $\eta_{\tilde\mu}>0$ is from the induction hypothesis.
By the induction hypothesis, there exists a map $\psi\in C^\infty_c(\Omega;\R^m)$ such that
\begin{equation}\label{property-v-function}
\left\{\begin{array}{l}
         \bar\mu+\nabla \psi=\bar{\tilde\mu}+\nabla \psi\in V^{rc}\quad\mbox{in $\Omega$}, \\
         \big||\{x\in\Omega\,|\,|\bar\mu+\nabla \psi(x)-\xi_j|<\eta\}|-\lambda_j|\Omega|\big|<\frac{\epsilon}{6(k+1)}|\Omega|\quad\mbox{for all $1\le j\le k$}, \\
         \big||\{x\in\Omega\,|\,|\bar\mu+\nabla \psi(x)-\tilde\xi_{k+1}|<\eta\}|-\tilde\lambda_{k+1}|\Omega|\big|<\frac{\epsilon}{6(k+1)}|\Omega|,\\
         \|\psi\|_{L^\infty(\Omega)}<\frac{\epsilon}{2}.
       \end{array}
\right.
\end{equation}
Set
\[
E_j=\{x\in\Omega\,|\,|\bar\mu+\nabla \psi(x)-\xi_j|<\eta\}\;\;(1\le j\le k),
\]
\[
\tilde{E}_{k+1}= \{x\in\Omega\,|\,|\bar\mu+\nabla \psi(x)-\tilde\xi_{k+1}|<\eta\},\quad F=\Omega\setminus(\cup_{j=1}^k E_j\cup\tilde{E}_{k+1});
\]
then from the choice of $\eta$, we see that $E_1,\cdots,E_k$ and $\tilde{E}_{k+1}$ are pairwise disjoint.
We now choose finitely many disjoint open cubes $Q_1,\cdots,Q_N\subset\subset\tilde{E}_{k+1}$, parallel to the axes, so that
\begin{equation}\label{truncation-tilde-E}
|\tilde{E}_{k+1}\setminus\cup_{i=1}^N Q_i|<\frac{\epsilon}{6}|\Omega|.
\end{equation}
Fix an index $i\in\{1,\cdots, N\}$, and set $\eta_i=\max_{x\in\bar{Q}_i}|\bar\mu+\nabla \psi(x)-\tilde{\xi}_{k+1}|<\eta$. Choose finitely many disjoint dyadic cubes $Q^1_i,\cdots,Q^{N_i}_i\subset Q_i$ with $|Q_i\setminus\cup_{j=1}^{N_i}Q^j_i|=0$ so small that
\begin{equation}\label{property-v-Lip}
|\nabla \psi(x)-\nabla \psi(y)|<\frac{\eta-\eta_i}{2}\quad\forall x,y\in\bar{Q}^{j}_i,\forall j=1,\cdots,N_i.
\end{equation}
Fix an index $j\in\{1,\cdots,N_i\}$. Let $x^j_i$ denote the center of the cube $Q^j_i$, and set $\xi^j_i=\bar\mu+\nabla \psi(x^j_i)\in V^{rc};$ then $|\xi^j_i-\tilde{\xi}_{k+1}|\le\eta_i$.
Since the matrix $L$ satisfies (\ref{assume-1}), $\mathrm{rank}(\xi_{k+1}-\xi_{k+2})=1$, and $\mathcal{L}(\xi_{k+1})=\mathcal{L}(\xi_{k+2})(=t)$, we can apply Theorem \ref{thm:rank-1} to the cube $Q^j_i$ and number $0<\tilde\lambda<1$ to obtain that for any given $\tau>0$, there exist a function $h^j_i\in C^\infty_c(Q^j_i;\R^m)$ and two disjoint open sets $Q^j_{i,k+1},Q^j_{i,k+2}\subset\subset Q^j_i$ satisfying
\begin{itemize}
\item[(a)] $h^j_i\in C^\infty_c(Q^j_i;\R^m)$, $\mathcal{L}(\nabla h^j_i)=0$ in $Q^j_i$,
\item[(b)] $\dist(\nabla h^j_i,[-\tilde\lambda(\xi_{k+1}-\xi_{k+2}),(1-\tilde\lambda)(\xi_{k+1}-\xi_{k+2})])<\tau$ in $Q^j_i$,
\item[(c)] $\nabla h^j_i(x)= \left\{\begin{array}{ll}
                                (1-\tilde\lambda)(\xi_{k+1}-\xi_{k+2}) & \mbox{$\forall x\in Q^j_{i,k+1}$}, \\
                                -\tilde\lambda(\xi_{k+1}-\xi_{k+2}) & \mbox{$\forall x\in Q^j_{i,k+2}$},
                              \end{array}\right.$
\item[(d)] $\big||Q^j_{i,k+1}|-\tilde\lambda|Q^j_i|\big|<\tau$, $\big||Q^j_{i,k+2}|-(1-\tilde\lambda)|Q^j_i|\big|<\tau$,
\item[(e)] $\|h^j_i\|_{L^\infty(Q^j_i)}<\tau$.
\end{itemize}
For our purpose, we choose
\begin{equation}\label{tau-choice}
0<\tau<\min\Big\{\eta,\frac{\epsilon}{3},\frac{\epsilon|\Omega|}{12(k+1)\sum_{i=1}^N N_i}\Big\}.
\end{equation}

We now define
\[
\varphi=\psi+\sum_{1\le i\le N,\,1\le j\le N_i} h^j_i\chi_{Q^j_i}\quad\mbox{in $\Omega$.}
\]
Let us check that $\varphi:\Omega\to\R^m$ is a desired function.
It is clear from the construction and (\ref{property-v-function}) that
\[
\mbox{$\varphi\in C^\infty_c(\Omega;\R^m)$.}
\]
Let $1\le i\le N$ and $1\le j\le N_i$. Note from (\ref{property-v-function}), (e) and (\ref{tau-choice}) that $|\varphi|<5\epsilon/6$ in $Q^j_i$; thus, from the definition of $\varphi$, we have
\[
\|\varphi\|_{L^\infty(\Omega)}<\epsilon.
\]
It follows from (a) and (\ref{property-v-function}) that for all $x\in Q^j_i$, we have $\mathcal{L}(\bar\mu+\nabla \varphi(x))=\mathcal{L}(\bar\mu+\nabla \psi(x))+\mathcal{L}(\nabla h^j_i(x))=\mathcal{L}(\bar\mu+\nabla \psi(x))=t$, i.e., $\bar\mu+\nabla \varphi(x)\in\Sigma_t$.
In addition,  we have from (\ref{property-v-Lip}), (b) and the choice $0<\tau<\eta<\eta_\mu$ that for all $x\in Q^j_i$,
\[
\begin{split}
\mathrm{dist} & (\bar\mu+\nabla \varphi(x),[\xi_{k+1},\xi_{k+2}]) \\
& =\dist(\bar\mu+\nabla \psi(x)-\xi^j_i+\xi^j_i-\tilde{\xi}_{k+1}+\tilde{\xi}_{k+1}+\nabla h^j_i(x),[\xi_{k+1},\xi_{k+2}]) \\
& \le |\bar\mu+\nabla \psi (x)-\xi^j_i|+|\xi^j_i-\tilde\xi_{k+1}|+\mathrm{dist}(\tilde\xi_{k+1}+\nabla h^j_i(x),[\xi_{k+1},\xi_{k+2}])  \\
& \le \frac{\eta-\eta_i}{2} +\eta_i +\mathrm{dist}(\nabla h^j_i (x),[-\tilde\lambda(\xi_{k+1}-\xi_{k+2}),(1-\tilde\lambda)(\xi_{k+1}-\xi_{k+2})]) \\
& <\eta+\tau<2\eta<\mathrm{dist}([\xi_{k+1},\xi_{k+2}],\partial|_{\Sigma_t}V^{rc}).
\end{split}
\]
Combining these two observations, we see that $\bar\mu+\nabla \varphi\in V^{rc}$ in $Q^j_i$, and thus from (\ref{property-v-function}) and  the definition of $u$, we have
\[
\bar\mu+\nabla \varphi\in V^{rc}\quad\mbox{in $\Omega$}.
\]

We now write
\[
G_l=\{x\in\Omega\,|\,|\bar\mu+\nabla \varphi(x)-\xi_l|<\eta\}\quad(1\le l\le k+2).
\]
Let $1\le i\le N$ and $1\le j \le N_i$.
By (c) and (\ref{property-v-Lip}), for all $x\in Q^j_{i,k+1}$, we have
\[
\begin{split}
|\bar\mu+ & \nabla \varphi  (x)-  \xi_{k+1}| =|\bar\mu+\nabla \psi(x)+\nabla h^j_i(x)-\xi_{k+1}| \\
& = |\bar\mu+\nabla \psi(x)-\xi^j_i+\xi^j_i-\tilde\xi_{k+1}+\tilde\xi_{k+1}+(1-\tilde\lambda)(\xi_{k+1}-\xi_{k+2})-\xi_{k+1}|\\
& = |\bar\mu+\nabla \psi(x)-\xi^j_i+\xi^j_i-\tilde\xi_{k+1}| \le \frac{\eta-\eta_i}{2}+\eta_i<\eta.
\end{split}
\]
Likewise,  we have $|\bar\mu+\nabla \varphi(x)-\xi_{k+2}|<\eta$ for all $x\in Q^j_{i,k+2}$.
Thus it follows from (\ref{property-v-function}), (d) and (\ref{tau-choice}) that for $1\le l\le k$,
\[
\begin{split}
\big| |G_l|-\lambda_l|\Omega|\big| & \le \big| |E_l|-\lambda_l|\Omega|\big|+\sum_{1\le i\le N,\,1\le j\le N_i}|Q^j_i\setminus(Q^j_{i,k+1}\cup Q^j_{i,k+2})| \\
& \le \frac{\epsilon|\Omega|}{6(k+1)}+2\tau\sum_{i=1}^N N_i<\frac{\epsilon|\Omega|}{3(k+1)}< \epsilon|\Omega|.
\end{split}
\]
It now remains to check that this inequality also holds for $l=k+1,k+2$.
Note from the above observation that
\[
\big| |G_{k+1}|-\lambda_{k+1}|\Omega| \big|
\le \left| \Big|\bigcup_{1\le i\le N,\,1\le j\le N_i}Q^j_{i,k+1}\Big|-\tilde\lambda \tilde\lambda_{k+1}|\Omega| \right|
\]
\[
+|F|+|\tilde{E}_{k+1}\setminus\cup_{i=1}^N Q_i|+\sum_{1\le i\le N,\,1\le j\le N_i}^N|Q^j_i\setminus(Q^j_{i,k+1}\cup Q^j_{i,k+2})|
\]
\[
=:I_1+I_2+I_3+I_4.
\]
By (d), (\ref{property-v-function}), (\ref{truncation-tilde-E}) and (\ref{tau-choice}), we can estimate:
\[
\begin{split}
I_1= & \Bigg| \Big|\bigcup_{1\le i\le N,\,1\le j\le N_i}Q^j_{i,k+1}\Big| - \tilde\lambda|\cup_{i=1}^N Q_i| -\tilde\lambda\tilde\lambda_{k+1}|\Omega|\\
& +\tilde\lambda|\tilde{E}_{k+1}|-\tilde\lambda|\tilde{E}_{k+1}| + \tilde\lambda|\cup_{i=1}^N Q_i| \Bigg|
\end{split}
\]
\[
\le \sum_{1\le i\le N,\,1\le j\le N_i}\big||Q^j_{i,k+1}|-\tilde\lambda|Q^j_i|\big| +\tilde\lambda\big| |\tilde{E}_{k+1}|-\tilde\lambda_{k+1}|\Omega| \big|+\tilde\lambda| \tilde{E}_{k+1}\setminus\cup_{i=1}^N Q_i |
\]
\[
< \tau\sum_{i=1}^N N_i+\frac{\tilde\lambda \epsilon|\Omega|}{6(k+1)} +\frac{\tilde\lambda\epsilon|\Omega|}{6}<\frac{\epsilon|\Omega|}{2},
\]
\[
\begin{split}
I_2+I_3+I_4 & \le \sum_{l=1}^k\big||E_l|-\lambda_l|\Omega| \big|+\big| |\tilde{E}_{k+1}|-\tilde\lambda_{k+1}|\Omega| \big|
+\frac{\epsilon|\Omega|}{6} + 2\tau\sum_{i=1}^N N_i\\
& <\frac{\epsilon|\Omega|}{6}+\frac{\epsilon|\Omega|}{6}+\frac{\epsilon|\Omega|}{6}=\frac{\epsilon|\Omega|}{2}.
\end{split}
\]
In all, we get $I_1+I_2+I_3+I_4<\epsilon|\Omega|$. In a similar manner, we also see that
\[
\big||G_{k+2}|-\lambda_{k+2}|\Omega| \big|<\epsilon|\Omega|.
\]

We have checked that the assertion holds for the laminate $\mu$ of order $k+1$, and the proof is now complete.
\end{proof}

As the last preparation for the proof of Theorems \ref{thm:main-1} and \ref{thm:main-2}, we improve the above lemma to deal with $C^1$ boundary data.

\begin{lem}\label{lem:pre-result}
Assume (\ref{assume-1}). Let $\Omega\subset\R^n$ be a bounded domain, let $U$ be a bounded open set in $\Sigma_t$, and let $v\in C^1(\bar\Omega;\R^m)$ be a map satisfying
\[
\nabla v(x)\in U^{rc}\;\;\mbox{for all $x\in\Omega$}.
\]
Then for each $\epsilon>0$, there exist a map $u\in C^1(\bar\Omega;\R^m)$ and an open set $\Omega'\subset\subset\Omega$ with $|\partial \Omega'|=0$ such that
\[\left\{
\begin{array}{ll}
  u(x)=v(x) & \mbox{for all $x$ near $\partial\Omega$}, \\
  \nabla u\in U^{rc} & \mbox{in $\Omega$}, \\
  \nabla u\in U & \mbox{in $\Omega'$}, \\
  |\Omega\setminus \Omega'|<\epsilon|\Omega|, \\
  \|u-v\|_{L^\infty(\Omega)}<\epsilon.
\end{array}\right.
\]
\end{lem}

\begin{proof}
Let $\epsilon>0$.
Choose finitely many disjoint open cubes $Q_1,\cdots, Q_{N}\subset\subset\Omega$, parallel to the axes, such that
\begin{equation}\label{set-truncation}
|\Omega \setminus\cup_{i=1}^{N}Q_i|<\frac{\epsilon}{3}|\Omega|.
\end{equation}
Fix an index $1\le i\le N$. As $\nabla v\in U^{rc}$ on $\bar{Q}_i$, we can use Corollary \ref{prop:sets} to choose an open set $V_i$ in $\Sigma_t$ with $\bar{V}_i\subset U$ such that
\[
\nabla v\in V_i^{rc}\;\;\mbox{on $\bar{Q}_i$}\quad\mbox{and}\quad \overline{(V_i^{rc})}\subset U^{rc}.
\]
Set
\[
\delta_i=\min\Big\{\mathrm{dist}(\partial|_{\Sigma_t}V_i^{rc},\partial|_{\Sigma_t}U^{rc}), \mathrm{dist}(\partial|_{\Sigma_t}V_i,\partial|_{\Sigma_t} U)\Big\}>0.
\]
Then divide $Q_i$ into finitely many disjoint dyadic cubes $Q_{i,1},\cdots,Q_{i,N_{i}}$ whose union has measure $|Q_i|$ and such that
\begin{equation}\label{grad-u-tilde-perturbation}
|\nabla v(x)-\nabla v(y)|<\frac{\delta_i}{2}
\end{equation}
for all $x,y\in Q_{i,j}$ and $j=1,\cdots, N_{i}.$
Now, fix an index $1\le j\le N_{i},$  let $x_{i,j}$ denote the center of the cube $Q_{i,j}$, and set $\xi_{i,j}=\nabla v(x_{i,j})\in V_i^{rc}.$ Then we apply Lemma \ref{lem:key-lemma} to obtain a map $\varphi_{i,j}\in C^\infty_c(Q_{i,j};\R^m)$ such that
\begin{equation}\label{phi-function-properties}
\left\{\begin{array}{l}
  \xi_{i,j}+\nabla\varphi_{i,j}\in V_i^{rc}\quad\mbox{in $Q_{i,j}$,} \\
  \big|\{x\in Q_{i,j}\,|\, \xi_{i,j}+\nabla\varphi_{i,j}(x)\not\in V_i\}\big|<\frac{\epsilon}{3} |Q_{i,j}|,\\
  \|\varphi_{i,j}\|_{L^\infty(Q_{i,j})}<\epsilon.
\end{array}\right.
\end{equation}

Define
\[
u=v+\sum_{1\le i\le N,\,1\le j\le N_{i}} \varphi_{i,j}\chi_{Q_{i,j}} \quad\mbox{in $\Omega$.}
\]
Then by (\ref{phi-function-properties}) and the definition of $u$, we have
\[
u\in C^1(\bar\Omega;\R^m),\;\; u(x)=v(x)\;\mbox{for all $x$ near $\partial\Omega$},\;\;\mbox{and}\;\;\|u-v\|_{L^\infty(\Omega)}<\epsilon.
\]
Let $1\le i\le N$ and $1\le j\le N_{i}$.
Then for all $x\in Q_{i,j}$, we have $\xi_{i,j}+\nabla \varphi_{i,j}(x)\in V_i^{rc}$ and
\begin{equation}\label{perturbation-inclusion}
|\nabla u(x)-(\xi_{i,j}+\nabla \varphi_{i,j}(x))|=|\nabla v(x)-\xi_{i,j}|<\frac{\delta_i}{2}\quad\mbox{(by (\ref{grad-u-tilde-perturbation}))};
\end{equation}
thus from the definition of $\delta_i$, we get $\nabla u(x)\in U^{rc}$. By the definition of $u$, we now see that
\[
\nabla u(x)\in U^{rc}\quad\mbox{for all $x\in\Omega$}.
\]
We write
\[
E_{i,j}=\{x\in Q_{i,j}\,|\, \xi_{i,j}+\nabla\varphi_{i,j}(x)\not\in V_i\},\quad G_{i,j}=Q_{i,j}\setminus E_{i,j};
\]
then $|E_{i,j}|<\frac{\epsilon}{3}|Q_{i,j}|$, and $G_{i,j}$ is an open set in $Q_{i,j}$.
If $x\in G_{i,j}$, then $\xi_{i,j}+\nabla\varphi_{i,j}(x)\in V_i$, and (\ref{perturbation-inclusion}) holds;
thus $\nabla u(x)\in U$. We now choose an open subset $H_{i,j}$ of $G_{i,j}$ such that
\begin{equation}\label{set-H-selection}
|G_{i,j}\setminus H_{i,j}|<\frac{\epsilon}{3}|Q_{i,j}|\quad\mbox{and}\quad |\partial H_{i,j}|=0.
\end{equation}
Let
\[
\Omega'=\bigcup_{1\le i\le N,\,1\le j\le N_{i}} H_{i,j};
\]
then $\Omega'\subset\subset\Omega$, $|\partial \Omega'|=0$, and
\[
\nabla u(x)\in U\quad\mbox{for all $x\in \Omega'$}.
\]
Moreover,  we have from (\ref{set-truncation}), (\ref{phi-function-properties}) and (\ref{set-H-selection})  that
\[
\begin{split}
|\Omega\setminus\Omega'|= & |\Omega\setminus\cup_{i=1}^{N}Q_i|+\sum_{1\le i\le N,\, 1\le j\le N_{i}} |E_{i,j}|\\
& +\sum_{1\le i\le N,\, 1\le j\le N_{i}} |G_{i,j}\setminus H_{i,j}|<\epsilon|\Omega|.
\end{split}
\]

The proof is now complete.
\end{proof}

\section{Proof of main theorems}\label{sec:proof-main-theorems}
We are now ready to prove Theorems \ref{thm:main-1} and \ref{thm:main-2} by iteration of the result of Lemma \ref{lem:pre-result} in suitable ways. We first finish the proof of Theorem \ref{thm:main-1} using a relatively simple iteration scheme.

\begin{proof}[Proof of Theorem \ref{thm:main-1}]
We only consider the case that $v\in C^1(\bar\Omega;\R^m)$ as the general case that $v$ is piecewise $C^1$ can be handled by following  the proof below for countably many disjoint open subsets on which $v$ is $C^1$ up to the boundary.     Now, by assumption, we have $\nabla v\in\overline{(U^{rc})}=U^{rc}\cup\partial|_{\Sigma_t} U^{rc}$ in $\Omega$ and $|\Gamma|=0$, where $\Gamma:=\{x\in\Omega\,|\, \nabla v(x)\in \partial|_{\Sigma_t} U^{rc}\}.$ So $\Omega':=\Omega\setminus\Gamma$ is an open subset of $\Omega$ with $|\Omega\setminus\Omega'|=0$. Fix an $0<\epsilon\ll 1$.

We write $\Omega^{(0)}=\Omega'$ and $\tilde{u}^{(0)}=v$ in $\Omega^{(0)}$. Since $\tilde{u}^{(0)}\in C^1(\bar{\Omega}^{(0)};\R^m)$ and $\nabla\tilde{u}^{(0)}\in U^{rc}$ in $\Omega^{(0)}$, we can apply Lemma \ref{lem:pre-result} to find a map $\tilde{u}^{(1)}\in C^1(\bar\Omega^{(0)};\R^m)$ and an open set $G^{(0)}\subset\subset\Omega^{(0)}$ with $|\partial G^{(0)}|=0$ such that setting $\Omega^{(1)}=\Omega^{(0)}\setminus\bar{G}^{(0)}$, we have
\[
\left\{
\begin{array}{ll}
    \tilde{u}^{(1)}(x)=\tilde{u}^{(0)}(x) & \mbox{for all $x$ near $\partial\Omega^{(0)}$}, \\
    \nabla \tilde{u}^{(1)}\in U^{rc} & \mbox{in $\Omega^{(0)}$}, \\
    \nabla \tilde{u}^{(1)}\in U & \mbox{in $G^{(0)}$}, \\
    |\Omega^{(1)}|<\epsilon|\Omega^{(0)}|, & \\
    \|\tilde{u}^{(1)}-\tilde{u}^{(0)}\|_{L^\infty(\Omega^{(0)})}<\frac{\epsilon}{2\cdot 2}.
  \end{array}\right.
\]
Since $\tilde{u}^{(1)}\in C^1(\bar{\Omega}^{(1)};\R^m)$ and $\nabla\tilde{u}^{(1)}\in U^{rc}$ in $\Omega^{(1)}$, we can also apply Lemma \ref{lem:pre-result} to obtain a map $\tilde{u}^{(2)}\in C^1(\bar\Omega^{(1)};\R^m)$ and an open set $G^{(1)}\subset\subset\Omega^{(1)}$ with $|\partial G^{(1)}|=0$ such that letting $\Omega^{(2)}=\Omega^{(1)}\setminus\bar{G}^{(1)}$, we have
\[
\left\{
\begin{array}{ll}
    \tilde{u}^{(2)}(x)=\tilde{u}^{(1)}(x) & \mbox{for all $x$ near $\partial\Omega^{(1)}$}, \\
    \nabla \tilde{u}^{(2)}\in U^{rc} & \mbox{in $\Omega^{(1)}$}, \\
    \nabla \tilde{u}^{(2)}\in U & \mbox{in $G^{(1)}$}, \\
    |\Omega^{(2)}|<\epsilon|\Omega^{(1)}|, & \\
    \|\tilde{u}^{(2)}-\tilde{u}^{(1)}\|_{L^\infty(\Omega^{(1)})}<\frac{\epsilon}{2\cdot 2^2}.
  \end{array}\right.
\]
Repeating this process indefinitely, we obtain a sequence of open sets $\Omega^{(0)}\supset\Omega^{(1)}\supset\Omega^{(2)}\supset\cdots$, a sequence of open sets $G^{(k)}\subset\subset\Omega^{(k)}$  with $|\partial G^{(k)}|=0$ $(k=0,1,2,\cdots)$, and a sequence of maps $\tilde{u}^{(k+1)}\in C^1(\bar\Omega^{(k)};\R^m)$ $(k=0,1,2,\cdots)$ such that for every integer $k\ge 0$,
\[
\left\{
\begin{array}{ll}
  \Omega^{(k+1)}=\Omega^{(k)}\setminus\bar{G}^{(k)}, &  \\
  \tilde{u}^{(k+1)}(x)=\tilde{u}^{(k)}(x) &  \mbox{for all $x$ near $\partial\Omega^{(k)}$}, \\
  \nabla \tilde{u}^{(k+1)}\in U^{rc} & \mbox{in $\Omega^{(k)}$}, \\
  \nabla \tilde{u}^{(k+1)}\in U & \mbox{in $G^{(k)}$}, \\
  |\Omega^{(k)}|<\epsilon^k|\Omega| & \mbox{if $k\ge 1$}, \\
  \|\tilde{u}^{(k+1)}-\tilde{u}^{(k)}\|_{L^\infty(\Omega^{(k)})}<\frac{\epsilon}{2\cdot 2^{k+1}}.
\end{array}
\right.
\]
Let
\[
u^{(1)}=\left\{\begin{array}{ll}
          \tilde{u}^{(1)} & \mbox{in $\Omega^{(0)}$}, \\
          v & \mbox{in $\Omega\setminus\Omega^{(0)}$},
        \end{array}\right.
\]
and for each $k\in\N$, define
\[
u^{(k+1)}=\sum_{j=1}^k \tilde{u}^{(j)}\chi_{\bar{G}^{(j-1)}}+\tilde{u}^{(k+1)}\chi_{\Omega^{(k)}}+v\chi_{\Omega\setminus\Omega^{(0)}}\;\;\mbox{in $\Omega$}.
\]
Then for all $k\in\N$, we have $u^{(k)}\in C^1(\bar\Omega;\R^m)$, $u^{(k)}=v$ near $\partial\Omega$, and $\nabla u^{(k)}\in U^{rc}$ a.e. in $\Omega.$
Let
\[
u=\sum_{j=1}^\infty \tilde{u}^{(j)}\chi_{\bar{G}^{(j-1)}}+v\chi_{\Omega\setminus\Omega^{(0)}}\;\;\mbox{in $\Omega$}.
\]
Since $u^{(k)}$ $(k\in\N)$ are uniformly Lipschitz in $\Omega$ and $u^{(k)}\to u$ a.e. in $\Omega$ as $k\to\infty$, it follows that $u\in v+W^{1,\infty}_0(\Omega;\R^m)$. Note also that $\nabla u\in U$  in $\cup_{j=0}^\infty G^{(j)}$, where $|\cup_{j=0}^\infty G^{(j)}|=|\Omega|$, and that
\[
\|u-v\|_{L^\infty(\Omega)}\le\sum_{k=0}^\infty \frac{\epsilon}{2\cdot 2^{k+1}}=\frac{\epsilon}{2}<\epsilon.
\]

The proof is now complete.
\end{proof}

The proof of Theorem \ref{thm:main-2} relies on a more subtle iteration of the result of Lemma \ref{lem:pre-result}. We remark that the result of Theorem \ref{thm:main-1} cannot be used directly to prove Theorem \ref{thm:main-2} (cf.  \cite[Proof of Theorem 1.3]{MSv1}).

\begin{proof}[Proof of Theorem \ref{thm:main-2}]
Again we only consider the case that $v\in C^1(\bar\Omega;\R^m)$ as the piecewise $C^1$ case can be adapted easily from the simpler case.

For each $j\in\N$, let
\[
\Omega_j=\{x\in\Omega\,|\, \mathrm{dist}(x,\partial\Omega)>2^{-j}\}.
\]
Let $\rho\in C^{\infty}_c(\R^n)$ denote the standard mollifier, and for each $\epsilon>0$, let $\rho_\epsilon(x)=\epsilon^{-n}\rho(x/\epsilon)$ for all $x\in\R^n$.

Let
\[
\Omega^{(1)}=\{x\in\Omega\,|\,\nabla v(x)\in U_1\};
\]
then $\Omega^{(1)}$ is an open subset of $\Omega$ with $|\Omega\setminus\Omega^{(1)}|=0$. Let us write $u^{(1)}=v$ in $\Omega^{(1)}$, and fix any two numbers $\epsilon>0$ and  $0<\delta_1< 1$.

Choose an $0<\epsilon_1<2^{-1}$ such that
\[
\|\rho_{\epsilon_1}\ast \nabla u^{(1)}-\nabla u^{(1)} \|_{L^\infty(\Omega_1)}<2^{-1}.
\]
Let
\[
\delta_2=\min\{2^{-2}\epsilon,\delta_1\epsilon_1/2\}.
\]
Since $\nabla u^{(1)}\in U_1\subset U_2^{rc}$ in $\Omega^{(1)}$, it follows from Lemma \ref{lem:pre-result} that there exist a map $u^{(2)}\in C^1(\bar\Omega^{(1)};\R^m)$ and an open set $\Omega^{(2)}\subset\subset\Omega^{(1)}$ with $|\partial\Omega^{(2)}|=0$ such that
\[\left\{
\begin{array}{ll}
  u^{(2)}(x)=u^{(1)}(x)=v(x) & \mbox{for all $x$ near $\partial\Omega^{(1)}$}, \\
  \nabla u^{(2)}\in U_2^{rc} & \mbox{in $\Omega^{(1)}$}, \\
  \nabla u^{(2)}\in U_2 & \mbox{in $\Omega^{(2)}$}, \\
  |\Omega^{(1)}\setminus \Omega^{(2)}|<\delta_2|\Omega^{(1)}|, \\
  \|u^{(2)}-u^{(1)}\|_{L^\infty(\Omega^{(1)})}<\delta_2.
\end{array}\right.
\]
Next, choose an $0<\epsilon_2<\min\{\epsilon_1,2^{-2}\}$ such that
\[
\|\rho_{\epsilon_2}\ast \nabla u^{(2)}-\nabla u^{(2)} \|_{L^\infty(\Omega_2)}<2^{-2}.
\]
Let
\[
\delta_3=\min\{2^{-3}\epsilon,\delta_2\epsilon_2/2\}.
\]
Since $\nabla u^{(2)}\in U_2^{rc}\subset U_3^{rc}$ in $\Omega^{(1)}$, it follows again from Lemma \ref{lem:pre-result} that there exist  a map $u^{(3)}\in C^1(\bar\Omega^{(1)};\R^m)$ and an open set $\Omega^{(3)}\subset\subset\Omega^{(1)}$ with $|\partial\Omega^{(3)}|=0$ such that
\[\left\{
\begin{array}{ll}
  u^{(3)}(x)=u^{(2)}(x)=v(x) & \mbox{for all $x$ near $\partial\Omega^{(1)}$}, \\
  \nabla u^{(3)}\in U_3^{rc} & \mbox{in $\Omega^{(1)}$}, \\
  \nabla u^{(3)}\in U_3 & \mbox{in $\Omega^{(3)}$}, \\
  |\Omega^{(1)}\setminus \Omega^{(3)}|<\delta_3|\Omega^{(1)}|, \\
  \|u^{(3)}-u^{(2)}\|_{L^\infty(\Omega^{(1)})}<\delta_3.
\end{array}\right.
\]
Repeating this process indefinitely,  we obtain a sequence $\{u^{(j)}\}_{j=2}^\infty$ in $C^1(\bar\Omega^{(1)};\R^m)$, a sequence of open sets $\Omega^{(j)}\subset\subset\Omega^{(1)}$ with $|\partial\Omega^{(j)}|=0$ $(j\ge 2)$, and a decreasing sequence $\{\epsilon_j\}_{j=1}^\infty$ in $(0,1/2)$ with $0<\epsilon_j<2^{-j}$ such that  for every integer $j\ge 2$, we have
\[\left\{
\begin{array}{l}
  \delta_j:=\min\{2^{-j}\epsilon,\delta_{j-1}\epsilon_{j-1}/2\},  \\
  \|\rho_{\epsilon_j}\ast \nabla u^{(j)}-\nabla u^{(j)} \|_{L^\infty(\Omega_j)}<2^{-j},  \\
  u^{(j)}(x)=v(x)\quad \mbox{for all $x$ near $\partial\Omega^{(1)}$}, \\
  \nabla u^{(j)}\in U_j^{rc}\quad \mbox{in $\Omega^{(1)}$}, \\
  \nabla u^{(j)}\in U_j\quad \mbox{in $\Omega^{(j)}$}, \\
  |\Omega^{(1)}\setminus \Omega^{(j)}|<\delta_j|\Omega|, \\
  \|u^{(j)}-u^{(j-1)}\|_{L^\infty(\Omega^{(1)})}<\delta_j.
\end{array}\right.
\]
We then extend $u^{(j)}\equiv v$ on $\Omega\setminus\Omega^{(1)}$ for all $j\ge 1$.

Since $\sum_{j=2}^\infty\delta_j\le\frac{\epsilon}{2}<\infty$ and $U^{rc}_j$ $(j\in\N)$ are uniformly bounded, we have
\[
\mbox{$\|u^{(j)}-u\|_{L^\infty(\Omega)}\to 0$\;\;\;for some\;\;$u\in v+W^{1,\infty}_0(\Omega;\R^m)$}
\]
and
\[
\|u-v\|_{L^\infty(\Omega)}\le\frac{\epsilon}{2}<\epsilon.
\]
It now remains to show that $\nabla u\in K$ a.e. in $\Omega$.
Note
\[
\begin{split}
\|\nabla u^{(j)}-\nabla u\|_{L^1(\Omega)}\le & \|\nabla u^{(j)}-\nabla u\|_{L^1(\Omega\setminus\Omega_j)} +\|\nabla u^{(j)}-\rho_{\epsilon_j}\ast\nabla u^{(j)}\|_{L^1(\Omega_j)} \\
& +\|\rho_{\epsilon_j}\ast(\nabla u^{(j)}-\nabla u)\|_{L^1(\Omega_j)} +\|\rho_{\epsilon_j}\ast\nabla u-\nabla u\|_{L^1(\Omega_j)} \\
=:&  I_{1,j}+I_{2,j}+I_{3,j}+I_{4,j}.
\end{split}
\]
As $j\to\infty$,
\[
\begin{split}
I_{1,j}\le & C|\Omega\setminus\Omega_j|\to 0, \\
I_{2,j}\le & 2^{-j}|\Omega|\to 0,\\
I_{3,j}\le & \|\nabla\rho_{\epsilon_j}\ast(u^{(j)}- u)\|_{L^1(\Omega_j)}\le \frac{C}{\epsilon_j}\sum_{i=j+1}^\infty \delta_i \le  \frac{C}{2}\sum_{i=j}^\infty\delta_i\to 0, \\
I_{4,j}\le & \|\rho_{\epsilon_j}\ast\nabla u-\nabla u\|_{L^1(\Omega)} \to 0,\;\;\mbox{with $\nabla u:=0$ outside $\Omega$}.
\end{split}
\]
Thus after passing to a subsequence if necessary, we have $\nabla u^{(j)}\to\nabla u$ a.e. in $\Omega.$  We now claim that for a.e. $x\in\Omega$, we have $\nabla u^{(j)}(x)\in U_j$ for infinitely many indices $j\in\N.$ Suppose on the contrary that there is a set $N\subset\Omega$ of positive measure such that for each $x\in N$, we have $\nabla u^{(j)}(x) \in U_j$ for only finitely many indices $j\in\N$. For each $k\in\N$, let
\[
N_k=\{x\in N\,|\, \nabla u^{(j)}(x)\not\in U_j\;\;\mbox{for all $j> k$}\};
\]
then $N_1\subset N_2\subset\cdots$ and $N=\cup_{k\in\N}N_k$. Choose a $k_0\in\N$ so large that $|N_{k_0}|\ge|N|/2>0.$ Then for all $x\in N_{k_0}$, we have $\nabla u^{(j)}(x)\not\in U_j$ for all $j>k_0$. By the construction above, we thus have
\[
N_{k_0}\subset \Omega\setminus\Omega^{(j)}\quad\forall j>k_0.
\]
As $|\Omega\setminus\Omega^{(j)}|<\delta_j\to 0$ as $j\to\infty$, we have $|N_{k_0}|=0$, a contradiction. Thus, for a.e. $x\in\Omega$, we have that $\nabla u^{(j)}(x)\to\nabla u(x)$ in $\Sigma_t$ and that there is an increasing sequence $\{j_k\}_{k\in\N}$ in $\N$ such that
\[
\nabla u^{(j_k)}(x)\in U_{j_k}\quad\forall k\in \N.
\]
Since $\{U_j\}_{j\in\N}$ is an in-approximation of $K$ in $\Sigma_t$, we now have for such an $x\in\Omega$ that
\[
\nabla u(x)\in K;
\]
hence $\nabla u\in K$ a.e. in $\Omega$.

The proof is now complete.
\end{proof}

\section{Proof of the applications}\label{sec:proof-applications}

In this section, we finish the proof of Theorem \ref{thm:eikonal} and Corollary \ref{coro:T4-configuration}.

\subsection{Proof of Theorem \ref{thm:eikonal}}

We first prove Theorem \ref{thm:eikonal} under an additional hypothesis that the linear map $L:\R^n\to\R^m$ is injective, where $m\ge n\ge 2$. We follow the notations in the beginning of the full proof below. Fix an integer $k_0\ge 1$ so large that
\[
\eta\in\{\xi\in U\,|\, \mathrm{dist}(\xi,K)>1/k_0\}=:U_1.
\]
For $k=2,3,\cdots,$ define
\[
U_k=\bigg\{\xi\in U\,|\, \frac{1}{k_0+k+1}<\mathrm{dist}(\xi,K)<\frac{1}{k_0+k}\bigg\}.
\]
Then it is easy to see that $\{U_k\}_{k\in\N}$ is an in-approximation of $K$ in $\Sigma_t$ with $\nabla v_{\eta,\gamma}=\eta\in U_1$. Therefore, the result follows from Theorem \ref{thm:main-2}. However, as explained in Introduction, we should perform a more careful justification for the general case without such an additional assumption. We adopt the Baire category framework for the proof below.

\begin{proof}[Proof of Theorem \ref{thm:eikonal}]
Set
\[
\Sigma_t=\{\xi\in\M^{m\times n}\,|\, \mathcal{L}(\xi)=t\}\;\;\mbox{and}\;\;B_t=\{\xi\in\Sigma_t\,|\,|\xi|<1\};
\]
then $\eta\in B_t$ and $\eta^\pm_{a\otimes b}\in\partial|_{\Sigma_t}B_t$.
For each $\alpha>0$, let
\[
V_\alpha=\Big\{\xi\in\Sigma_t\,\big|\, \mathrm{dist}\big(\xi,\overline{\eta^+_{a\otimes b}\eta^-_{a\otimes b}}\big)<\alpha\Big\},
\]
where $\overline{\eta^+_{a\otimes b}\eta^-_{a\otimes b}}$ denotes the straight line in $\Sigma_t$ passing through $\eta^\pm_{a\otimes b}$. Choose an $\alpha_\epsilon>0$ so small that $\bar{V}_{\alpha_\epsilon}\cap\partial|_{\Sigma_t}B_t$ is the disjoint union of two connected sets $K^\pm$ with $\eta^\pm_{a\otimes b}\in K^\pm$ such that $\mathrm{diam}(K^\pm)<\epsilon/2$. Then set
\[
K=K^+\cup K^-\;\;\mbox{and}\;\;U=V_{\alpha_\epsilon}\cap B_t.
\]
Define the \emph{admissible class} $\mathcal{A}$ as
\[
\mathcal{A}=\{v\in v_{\eta,\gamma}+C^\infty_c(\Omega;\R^m)\,|\,\mbox{$\nabla v\in U$ in $\Omega$, $\|v-v_{\eta,\gamma}\|_{L^\infty(\Omega)}<\epsilon/2$}\};
\]
then  $v_{\eta,\gamma}\in\mathcal{A}\neq\emptyset$.
For each $\delta>0$, define the \emph{$\delta$-approximating class} $\mathcal{A}_\delta$ as
\[
\mathcal{A}_\delta=\bigg\{v\in \mathcal{A}\,\Big|\,\int_{\Omega}\mathrm{dist}(\nabla v(x),K)\,dx<\delta|\Omega|\bigg\}.
\]

We now divide the proof into several steps as follows.

\underline{\textbf{Claim:}} For each $\delta>0$,
\[
\mbox{$\mathcal{A}_\delta$ is dense in $\mathcal{A}$ with respect to the $L^\infty(\Omega;\R^m)$-norm.}
\]

Suppose for the moment that Claim holds. We now generate solutions to problem (\ref{eikonal-problem}) under the Baire category framework.

\underline{\textbf{Baire's category method:}}
Let $\mathcal{X}$ denote the closure of $\mathcal A$ in the space $L^\infty(\Omega;\R^m)$. Then $(\mathcal{X},\|\cdot\|_{L^\infty(\Omega)})$ is a nonempty complete metric space.  As $U$ is bounded in $\Sigma_t$, we easily see that
\[
\mathcal{X}\subset v_{\eta,\gamma}+W^{1,\infty}_0(\Omega;\R^m)
\]
and that $\|u-v_{\eta,\gamma}\|_{L^\infty(\Omega)}\le\epsilon/2<\epsilon$ for all $u\in\mathcal{X}$.
Since the gradient operator $\nabla:\mathcal{X}\to L^1(\Omega;\M^{m\times n})$ is a Baire-one map \cite[Proposition 10.17]{Da}, it follows from the Baire Category Theorem \cite[Theorem 10.15]{Da} that the set $\mathcal{C}_{\nabla}$ of points of continuity for the operator $\nabla$ is dense in $\mathcal{X}$.

\underline{\textbf{Solution set $\mathcal{C}_\nabla$:}} We now check that every map $u\in\mathcal{C}_\nabla$ is a solution to problem (\ref{eikonal-problem}). Let $u\in\mathcal{C}_\nabla.$ From the previous step, we have
\begin{equation}\label{app-1}
\|u-v_{\eta,\gamma}\|_{L^\infty(\Omega)}<\epsilon.
\end{equation}
By the definition of $\mathcal{X}$, we can choose a sequence $\{\tilde u_k\}_{k\in\N}$ in $\mathcal A$ such that $\|\tilde u_k-u\|_{L^\infty(\Omega)}\to 0$ as $k\to\infty$. By the result of Claim above, for each $k\in\N$, we can choose a map $u_k\in\mathcal{A}_{1/k}$ such that $\|u_k-\tilde{u}_k\|_{L^\infty(\Omega)}<1/k;$ thus $\|u_k-u\|_{L^\infty(\Omega)}\to 0$. Since $u\in\mathcal{C}_\nabla$, we now have $\nabla u_k\to\nabla u$ in $L^1(\Omega;\M^{m\times n})$, and so $\nabla u_k(x)\to\nabla u(x)$ in $\M^{m\times n}$  for  a.e. $x\in\Omega$ after passing to a subsequence if necessary. On the other hand, from $u_k\in\mathcal{A}_{1/k}$, we have
\[
\int_\Omega\mathrm{dist}(\nabla u_k(x),K)\, dx<\frac{1}{k}|\Omega|\to 0.
\]
Since $K$ is compact and $U$ is bounded, it follows from the Dominate Convergence Theorem  that the map $u\in v_{\eta,\gamma}+W_0^{1,\infty}(\Omega;\R^m)$ satisfies the differential inclusion
\[
\nabla u\in K\quad\mbox{a.e. in $\Omega$}.
\]
This together with (\ref{app-1}) implies that $u$ is a solution to (\ref{eikonal-problem}).

\underline{\textbf{Infinitely many solutions:}} To show that there are infinitely many solutions to problem (\ref{eikonal-problem}), it now suffices to check that $\mathcal{C}_\nabla$ has infinitely many elements. Suppose on the contrary that $\mathcal{C}_\nabla$ has only finitely many elements. Since $\mathcal{C}_\nabla$ is dense in $\mathcal X$, we thus have
\[
v_{\eta,\gamma}\in\mathcal{X}=\overline{\mathcal{C}_\nabla}=\mathcal{C}_\nabla.
\]
By the previous step, we arrive at the conclusion that $v_{\eta,\gamma}$ is a solution to (\ref{eikonal-problem}), a contradiction. Therefore, $\mathcal{C}_\nabla$ has infinitely many elements.

\underline{\textbf{Proof of Claim:}} To finish the proof, it only remains to verify Claim above. Let $\delta>0$, $\theta>0$ and $v\in\mathcal{A}$. We will  show that there exists a map $v_\theta\in\mathcal{A}_\delta$ such that $\|v_\theta-v\|_{L^\infty(\Omega)}<\theta$.

Since $v\in\mathcal{A}$, there exists a map $\varphi\in C^\infty_c(\Omega;\R^m)$ such that $v=v_{\eta,\gamma}+\varphi$ in $\Omega$,  $\|\varphi\|_{L^\infty(\Omega)}<\epsilon/2$, and $\eta+\nabla\varphi\in U$ in $\Omega$. Set
\begin{equation}\label{app-2-epsilon-prime-def}
\epsilon'=2^{-1}(2^{-1}\epsilon-\|\varphi\|_{L^\infty(\Omega)})>0.
\end{equation}
Choose finitely many open cubes $Q_1,\cdots, Q_N\subset\subset\Omega,$ parallel to the axes, such that
\begin{equation}\label{app-3-first-cube-choice}
|\Omega\setminus\cup_{i=1}^N Q_i|<\frac{\delta|\Omega|}{4}.
\end{equation}

Fix an index $1\le i\le N.$ Let us write
\begin{equation}\label{app-4-positive-distance}
d_i=\min_{\bar{Q}_i}\mathrm{dist}(\eta+\nabla\varphi,\partial|_{\Sigma_t}U)>0.
\end{equation}
Choose finitely many dyadic cubes $Q_{i,1},\cdots, Q_{i,N_i}\subset Q_i$ with $|Q_i\setminus\cup_{j=1}^{N_i}Q_{i,j}|=0$ such that
\begin{equation}\label{app-5-grad-pertubation}
|\nabla\varphi(x)-\nabla\varphi(x)|<\min\Big\{\frac{d_i}{2},\frac{\delta}{16}\Big\}
\end{equation}
for all $x,y\in\bar{Q}_{i,j}$ and all $1\le j\le N_i.$ For each $1\le j\le N_i$, let $x_{i,j}$ denote the center of the cube $Q_{i,j}$ and write $\xi_{i,j}=\nabla v(x_{i,j})=\eta+\nabla\varphi(x_{i,j})\in U$. Define
\begin{equation}\label{app-6-Lambda-index-set}
\Lambda_i=\{j\in\{1,\cdots,N_i\}\,|\,\mathrm{dist}(\xi_{i,j},K)>\delta/8\}.
\end{equation}
Fix an index $j\in\Lambda_i$. Then we can choose two numbers $s^+_{i,j}>0>s^-_{i,j}$ such that
\begin{equation}\label{app-7-artificial-distance-choice}
\xi_{i,j}+s^\pm_{i,j}a\otimes b\in U\quad\mbox{and}\quad \mathrm{dist}(\xi_{i,j}+s^\pm_{i,j}a\otimes b,K^\pm)=\frac{\delta}{8}.
\end{equation}
Now, thanks to Theorem \ref{thm:rank-1}, for any given $\tau>0$, we can choose a map $\psi_{i,j}\in C^\infty_c(Q_{i,j};\R^m)$ and two disjoint open sets $Q^\pm_{i,j}\subset\subset Q_{i,j}$ satisfying the following:
\begin{itemize}
\item[(a)] $\mathcal{L}(\nabla\psi_{i,j})=0$ in $Q_{i,j}$,\\
\item[(b)] $\dist\big(\nabla\psi_{i,j},[-\lambda_{i,j}(s^+_{i,j}-s^-_{i,j})a\otimes b,(1-\lambda_{i,j})(s^+_{i,j}-s^-_{i,j})a\otimes b]\big)<\tau$ in $Q_{i,j}$,
\item[(c)] $\nabla \psi_{i,j}(x)= \left\{\begin{array}{ll}
                                (1-\lambda_{i,j})(s^+_{i,j}-s^-_{i,j})a\otimes b & \mbox{$\forall x\in Q^+_{i,j}$}, \\
                                -\lambda_{i,j}(s^+_{i,j}-s^-_{i,j})a\otimes b & \mbox{$\forall x\in Q^-_{i,j}$},
                              \end{array}\right.$
\item[(d)] $\big||Q^+_{i,j}|-\lambda_{i,j}|Q_{i,j}|\big|<\tau$, $\big||Q^-_{i,j}|-(1-\lambda_{i,j})|Q_{i,j}|\big|<\tau$,
\item[(e)] $\|\psi_{i,j}\|_{L^\infty(Q_{i,j})}<\tau$,
\end{itemize}
where $\lambda_{i,j}:=\frac{-s^-_{i,j}}{s^+_{i,j}-s^-_{i,j}}\in(0,1).$ Here, we choose
\begin{equation}\label{app-8-tau-choice}
0<\tau<\min\Big\{\frac{d_i}{2},\frac{\delta}{16},\theta,\epsilon',\frac{\delta|Q_{i,j}|}{8}\Big\}.
\end{equation}

Define
\[
v_\theta=v_{\eta,\gamma}+\varphi+\sum_{1\le i\le N,\,j\in\Lambda_i}\psi_{i,j}\chi_{Q_{i,j}}\;\;\mbox{in $\Omega$}.
\]
We now check that $v_\theta$ is a desired map for the proof of Claim.
By definition, we have
\[
v_\theta\in v_{\eta,\gamma}+C^\infty_c(\Omega;\R^m).
\]
Note from the definition of $v_\theta$, (e),  (\ref{app-8-tau-choice}) and (\ref{app-2-epsilon-prime-def}) that
\[
\|v_\theta-v_{\eta,\gamma}\|_{L^\infty(\Omega)}\le\|\varphi\|_{L^\infty(\Omega)}+\epsilon' <\frac{\epsilon}{2}.
\]
Let $1\le i\le N$, $j\in\Lambda_i$, and $x\in Q_{i,j}.$
For all $s^-_{i,j}\le s\le s^+_{i,j}$, we have from (\ref{app-5-grad-pertubation}), (b) and (\ref{app-8-tau-choice}) that
\[
\begin{split}
|\eta+ & \nabla\varphi(x)+\nabla\psi_{i,j}(x)-(\xi_{i,j}+s a\otimes b)| \\
& \le |\nabla\varphi(x)-\nabla\varphi(x_{i,j})|+|\nabla\psi_{i,j}(x)-s a\otimes b| \\
& <\min\Big\{\frac{d_i}{2},\frac{\delta}{16}\Big\}+\mathrm{dist}(\nabla\psi_{i,j}(x),[s^+_{i,j}a\otimes b,s^-_{i,j}a\otimes b])<\min\Big\{d_i,\frac{\delta}{8}\Big\}.
\end{split}
\]
This implies that
\[
\mathrm{dist}(\nabla v_\theta(x),[\xi_{i,j}+s^+_{i,j} a\otimes b,\xi_{i,j}+s^-_{i,j} a\otimes b])<\min\Big\{d_i,\frac{\delta}{8}\Big\}.
\]
Also, from (a), we have $\mathcal{L}(\nabla v_\theta(x))=\mathcal{L}(\eta+\nabla\varphi(x))=t$; hence, $\nabla v_\theta(x)\in\Sigma_t$.
Thus, from these two observations together with (\ref{app-4-positive-distance}) and (\ref{app-7-artificial-distance-choice}), we have $\nabla v_\theta(x)\in U$. By the definition of $v_\theta$, we now have
\[
\nabla v_\theta\in U\quad\mbox{in $\Omega$};
\]
therefore,  $v_\theta\in\mathcal{A}.$ Also, from (e) and (\ref{app-8-tau-choice}), we get
\[
\|v_\theta-v\|_{L^\infty(\Omega)}<\tau<\theta.
\]
Next, observe
\[
\begin{split}
\int_\Omega & \mathrm{dist}(\nabla v_\theta(x),K)\,dx= \int_{\Omega\setminus\cup_{i=1}^N Q_i}\mathrm{dist}(\nabla v(x),K)\,dx \\
& +\sum_{1\le i\le N,\,1\le j\le N_i,\,j\not\in\Lambda_i} \int_{Q_{i,j}}\mathrm{dist}(\nabla v(x),K)\,dx \\
& +\sum_{1\le i\le N,\,j\in\Lambda_i} \int_{Q_{i,j}\setminus(Q^+_{i,j}\cup Q^-_{i,j})}\mathrm{dist}(\eta+\nabla\varphi(x)+\nabla\psi_{i,j}(x),K)\,dx \\
& +\sum_{1\le i\le N,\,j\in\Lambda_i} \int_{Q^+_{i,j}\cup Q^-_{i,j}}\mathrm{dist}(\eta+\nabla\varphi(x)+\nabla\psi_{i,j}(x),K)\,dx \\
& =:I_1+I_2+I_3+I_4.
\end{split}
\]
Since $\nabla v_\theta\in U$ in $\Omega$, we have $\mathrm{dist}(\nabla v_\theta,K)\le 1$ in $\Omega$.
We now estimate:
\[
\begin{split}
I_1 & \le |\Omega\setminus\cup_{i=1}^N Q_i|<\frac{\delta|\Omega|}{4},\quad\mbox{(by (\ref{app-3-first-cube-choice}))} \\
I_2 & \le \sum_{1\le i\le N,\,1\le j\le N_i,\,j\not\in\Lambda_i} \frac{3\delta|Q_{i,j}|}{16}<\frac{\delta|\Omega|}{4},\quad\mbox{(by (\ref{app-5-grad-pertubation}) and (\ref{app-6-Lambda-index-set}))} \\
I_3 & \le \sum_{1\le i\le N,\,j\in\Lambda_i} |Q_{i,j}\setminus(Q^+_{i,j}\cup Q^-_{i,j})|<\frac{\delta|\Omega|}{4},\quad\mbox{(by (d) and (\ref{app-8-tau-choice}))} \\
I_4 & \le \sum_{1\le i\le N,\,j\in\Lambda_i} \frac{3\delta|Q_{i,j}|}{16}<\frac{\delta|\Omega|}{4}; \quad\mbox{(by (c), (\ref{app-5-grad-pertubation}) and (\ref{app-7-artificial-distance-choice}))}
\end{split}
\]
thus $I_1+I_2+I_3+I_4<\delta|\Omega|$. Therefore, we have $v_\theta\in\mathcal{A}_\delta$, and the proof of Claim is complete.

The theorem is now proved.
\end{proof}

\subsection{Proof of Corollary \ref{coro:T4-configuration}}

Let $A_1,A_2,A_3,A_4$ be the matrices defined in (\ref{app2-1-K-set-definition}). For $i=1,2,3,4,$ let $B_i$ be the open ball in the space $\Sigma_0=\{\xi\in\M^{2\times 2}\,|\,\mathcal{L}(\xi)=0\}$ with center $A_i$ and radius $\epsilon.$ Set $U=\cup_{i=1}^4 B_i$. Since $\nabla v_{\eta,\gamma}=\eta\in K^{rc}\subset U^{rc}$, we can apply Theorem \ref{thm:main-1} to obtain a map $u\in v_{\eta,\gamma}+ W_0^{1,\infty}(\Omega;\R^2)$ satisfying the conclusion of the corollary.

\section{Proof of Theorem \ref{thm:rank-1}}\label{sec:proof-rank-1}

We simply repeat the proof in \cite{KK1} without any modification.

Set $r=\rank(L).$ By (\ref{rank-1-1}), we have $1\le r\le m\wedge n=:\min\{m,n\}.$

\textbf{(Case 1):} Assume that the matrix $L$ satisfies
\[
\begin{split}
L_{ij}=0\;\;& \mbox{for all $1\le i\le m,\, 1\le j\le n$ but possibly the pairs}\\
& \mbox{$(1,1),(1,2),\cdots,(1,n),(2,2),\cdots,(r,r)$ of $(i,j)$};
\end{split}
\]
hence $L$ is of the form
\begin{equation}\label{rank-1-5}
L=\begin{pmatrix} L_{11} & L_{12} & \cdots & L_{1r} &  \cdots & L_{1n}\\
                    & L_{22} & & & & & \\
                    & & \ddots & & & & \\
                    & & & L_{rr} & & & \\
                    & & & & & & \end{pmatrix}\in\M^{m\times n}
\end{equation}
and that
\[
A-B=a\otimes e_1\;\;\mbox{for some nonzero vector $a=(a_1,\cdots,a_m)\in\R^m$},
\]
where each blank component in (\ref{rank-1-5}) is zero.
From (\ref{rank-1-1}) and $\rank(L)=r$, it follows that the product $L_{11}\cdots L_{rr}\ne 0$. Since $0=\mathcal L(A-B)=\mathcal L(a\otimes e_1)=L_{11}a_1$, we have $a_1=0$.

In this case, the linear map $\mathcal L:\M^{m\times n}\to\R$ is given by
\[
\mathcal L(\xi)=\sum_{j=1}^n L_{1j}\xi_{1j}+\sum_{i=2}^r L_{ii}\xi_{ii},\quad \xi\in\M^{m\times n}.
\]
We will find a linear differential operator $\Phi:C^1(\R^n;\R^m)\to C^0(\R^n;\R^m)$ such that
\begin{equation}\label{rank-1-6}
\mathcal L(\nabla\Phi v)\equiv 0 \quad\forall v\in C^2(\R^n;\R^m).
\end{equation}
So our candidate for such a $\Phi=(\Phi^1,\cdots,\Phi^m)$ is of the form
\begin{equation}\label{rank-1-7}
\Phi^i v=\sum_{1\le k\le m,\,1\le l\le n}a^i_{kl}v^k_{x_l},
\end{equation}
where $1\le i\le m$, $v\in C^1(\R^n;\R^m)$, and $a^i_{kl}$'s are real constants to be determined; then for $v\in C^2 (\R^n;\R^m)$, $1\le i\le m$, and $1\le j\le n$,
\[
\partial_{x_j}\Phi^i v =\sum_{1\le k\le m,\,1\le l\le n}a^i_{kl}v^k_{x_lx_j}.
\]
Rewriting (\ref{rank-1-6}) with this form of $\nabla\Phi v$ for $v\in C^2 (\R^n;\R^m)$, we have
\[
\begin{split}
0 & \equiv  \sum_{1\le k\le m,\,1\le j,l\le n} L_{1j}a^1_{kl}v^k_{x_lx_j} + \sum_{i=2}^r\sum_{1\le k\le m,\,1\le l\le n} L_{ii}a^i_{kl}v^k_{x_lx_i}\\
& = \sum_{k=1}^m \Big(L_{11}a^1_{k1}v^k_{x_1x_1}+\sum_{j=2}^r (L_{1j}a^1_{kj}+L_{jj}a^j_{kj})v^k_{x_jx_j}+\sum_{j=r+1}^n L_{1j}a^1_{kj}v^k_{x_jx_j} \\
& \quad +\sum_{l=2}^r (L_{11}a^1_{kl}+L_{1l}a^1_{k1}+L_{ll}a^l_{k1})v^k_{x_lx_1} +\sum_{l=r+1}^n (L_{11}a^1_{kl}+L_{1l}a^1_{k1})v^k_{x_lx_1} \\
& \quad +\sum_{2\le j<l\le r} (L_{1j}a^1_{kl}+L_{1l}a^1_{kj}+L_{jj}a^j_{kl}+L_{ll}a^l_{kj})v^k_{x_lx_j}\\
& \quad +\sum_{2\le j\le r,\,r+1\le l\le n} (L_{1j}a^1_{kl}+L_{1l}a^1_{kj}+L_{jj}a^j_{kl})v^k_{x_lx_j}\\
& \quad+\sum_{r+1\le j<l\le n} (L_{1j}a^1_{kl}+L_{1l}a^1_{kj})v^k_{x_lx_j} \Big).
\end{split}
\]

Should (\ref{rank-1-6}) hold, it  is thus sufficient to solve the following algebraic system for each $k=1,\cdots,m$ (after adjusting the letters for some indices):
\begin{eqnarray}
\label{rr-1}& L_{11}a^1_{k1}=0, &\\
\label{rr-4}& L_{1j}a^1_{kj}+L_{jj}a^j_{kj}=0 & \;\,\forall j=2,\cdots, r,\\
\label{rr-3}& L_{11}a^1_{kj}+L_{1j}a^1_{k1}+L_{jj}a^j_{k1}=0 & \;\,\forall j=2,\cdots, r, \\
\label{rr-5}& L_{1l}a^1_{kj}+L_{1j}a^1_{kl}+L_{ll}a^l_{kj}+L_{jj}a^j_{kl}=0 & \begin{array}{l}
                                                                    \forall j=3,\cdots, r, \\
                                                                    \forall l=2,\cdots, j-1,
                                                                  \end{array} \\
\label{rr-6}& L_{1j}a^1_{kj}=0 & \;\,\forall j=r+1,\cdots, n,\\
\label{rr-7}& L_{11}a^1_{kj}+L_{1j}a^1_{k1}=0 & \;\,\forall j=r+1,\cdots, n, \\
\label{rr-8}& L_{1l}a^1_{kj}+L_{1j}a^1_{kl}+L_{ll}a^l_{kj}=0 & \begin{array}{l}
                                                                    \forall j=r+1,\cdots, n, \\
                                                                    \forall l=2,\cdots, r,
                                                                  \end{array} \\
\label{rr-2}& L_{1l}a^1_{kj}+L_{1j}a^1_{kl}=0 & \begin{array}{l}
                                                                    \forall j=r+2,\cdots, n, \\
                                                                    \forall l=r+1,\cdots,j-1.
                                                                  \end{array}
\end{eqnarray}
Although these systems have infinitely many solutions, we will solve those in a way for a later purpose that the matrix $(a^j_{k1})_{2\le j, k\le m}\in\M^{(m-1)\times(m-1)}$ fulfills
\begin{equation}\label{rank-1-8}
a^j_{21}=a_j \quad\forall j=2,\cdots, m,\quad\mbox{and}\quad a^j_{k1}=0\quad\mbox{otherwise}.
\end{equation}
Firstly, we let the coefficients $a^i_{kl}\;(1\le i,k\le m,\,1\le l\le n)$ that do not appear in systems (\ref{rr-1})--(\ref{rr-2}) $(k=1,\cdots, m)$ be zero with an exception that we set $a^j_{21}=a_j$ for $j=r+1,\cdots,m$ to reflect (\ref{rank-1-8}). Secondly, for $1\le k\le m,\,k\ne 2$, let us take the trivial (i.e., zero) solution of system (\ref{rr-1})--(\ref{rr-2}). Finally, we take $k=2$ and solve system (\ref{rr-1})--(\ref{rr-2})  as follows with (\ref{rank-1-8}) satisfied.
Since $L_{11}\ne 0$, we set $a^1_{21}=0$; then (\ref{rr-1}) is satisfied. So we set
\[
a^j_{21}=-\frac{L_{11}}{L_{jj}}a^1_{2j},\;\;a^1_{2j}=-\frac{L_{jj}}{L_{11}}a_j \quad \forall j=2,\cdots,r;
\]
then (\ref{rr-3}) and (\ref{rank-1-8}) hold. Next, set
\[
a^j_{2j}=-\frac{L_{1j}}{L_{jj}}a^1_{2j}=\frac{L_{1j}}{L_{11}}a_j \quad\forall j=2,\cdots, r;
\]
then (\ref{rr-4}) is fulfilled. Set
\[
a^l_{2j}=-\frac{L_{1l}a^1_{2j}+L_{1j}a^1_{2l}}{L_{ll}}=\frac{L_{1l}L_{jj}a_j+L_{1j}L_{ll}a_l}{L_{ll}L_{11}},\;\;a^j_{2l}=0
\]
for $j=3,\cdots,r$ and $l=2,\cdots, j-1$; then (\ref{rr-5}) holds. Set
\[
a^1_{2j}=0\quad\forall j=r+1,\cdots,n;
\]
then (\ref{rr-6}) and (\ref{rr-7}) are satisfied. Lastly, set
\[
a^1_{2j}=0,\;\; a^l_{2j}=-\frac{L_{1j}}{L_{ll}}a^1_{2l}=\frac{L_{1j}}{L_{11}}a_l\quad\forall j=r+1,\cdots, n,\,\forall l=2,\cdots, r;
\]
then (\ref{rr-8}) and (\ref{rr-2}) hold. In summary, we have determined the coefficients $a^i_{kl}\;(1\le i,k\le m,\,1\le l\le n)$ in such a way that   system (\ref{rr-1})--(\ref{rr-2}) holds for each $k=1,\cdots, m$ and that (\ref{rank-1-8}) is also satisfied. Therefore, (1) follows from (\ref{rank-1-6}) and (\ref{rank-1-7}).

To prove (2), without loss of generality, we can assume $\Omega=(0,1)^n\subset\R^n.$ Let $\tau>0$ be given. Let $u=(u^1,\cdots,u^m)\in C^\infty(\Omega;\R^m)$ be a function to be determined. Suppose $u$ depends only on the first variable $x_1\in(0,1).$ We wish to have
\[
\nabla\Phi u(x)\in\{-\lambda a\otimes e_1,(1-\lambda) a\otimes e_1\}
\]
for all $x\in\Omega$ except in a set of small measure. Since $u(x)=u(x_1)$, it follows from (\ref{rank-1-7}) that for $1\le i\le m$ and $1\le j\le n$,
\[
\Phi^i u=\sum_{k=1}^m a^i_{k1} u^k_{x_1};\;\;\mbox{thus}\;\;\partial_{x_j}\Phi^i u=\sum_{k=1}^m a^i_{k1} u^k_{x_1 x_j}.
\]
As $a^1_{k1}=0$ for $k=1,\cdots, m$, we have $\partial_{x_j}\Phi^1 u =\sum_{k=1}^m a^1_{k1} u^k_{x_1 x_j}=0$ for $j=1,\cdots,n$. We first set $u^1\equiv 0$ in $\Omega$. Then from (\ref{rank-1-8}), it follows that for $i=2,\cdots, m$,
\[
\partial_{x_j}\Phi^i u =\sum_{k=2}^m a^i_{k1} u^k_{x_1 x_j}=a^i_{21}  u^2_{x_1 x_j}=a_i u^2_{x_1 x_j} = \left\{\begin{array}{ll}
                                  a_i u^2_{x_1 x_1} & \mbox{if $j=1$,} \\
                                  0 & \mbox{if $j=2,\cdots, n$.}
                                \end{array} \right.
\]
As $a_1=0$, we thus have that for $x\in\Omega$,
\[
\nabla\Phi u(x)=(u^2)''(x_1) a\otimes e_1.
\]
For irrelevant components of $u$, we simply take $u^3=\cdots =u^m\equiv 0$ in $\Omega$. Lastly, for a number $\delta>0$ to be chosen later, we choose a function $u^2(x_1)\in C^\infty_c(0,1)$ such that there exist two disjoint open sets $I_1,I_2\subset\subset (0,1)$ satisfying $\big||I_1|-\lambda\big|<\tau/2$, $\big||I_2|-(1-\lambda)\big|<\tau/2$, $\|u^2\|_{L^\infty(0,1)}<\delta$, $\|(u^2)'\|_{L^\infty(0,1)}<\delta$, $-\lambda\le (u^2)''(x_1)\le 1-\lambda$ for $x_1\in(0,1)$, and
\[
(u^2)''(x_1)= \left\{\begin{array}{ll}
                       1-\lambda & \mbox{if $x_1\in I_1$}, \\
                       -\lambda & \mbox{if $x_1\in I_2$}.
                     \end{array}
 \right.
\]
In particular,
\begin{equation}\label{rank-1-9}
\nabla \Phi u(x)\in[-\lambda a\otimes e_1,(1-\lambda)a\otimes e_1]\;\;\forall x\in\Omega.
\end{equation}
We now choose an open set $\Omega'_\tau\subset\subset\Omega':=(0,1)^{n-1}$ with $|\Omega'\setminus\Omega'_\tau|<\tau/2$ and a function $\eta\in C^\infty_c(\Omega')$ so that
\[
0\le\eta\le 1\;\;\mbox{in}\;\;\Omega',\;\; \eta\equiv 1\;\;\mbox{in}\;\Omega'_\tau,\;\;\mbox{and}\;\;|\nabla^i_{x'}\eta|<\frac{C}{\tau^i}\;\;(i=1,2)\;\;\mbox{in}\;\Omega',
\]
where $x'=(x_2,\cdots,x_n)\in\Omega'$ and the constant $C>0$ is independent of $\tau$.
Now, we define $g(x)=\eta(x') u(x_1)\in C^\infty_c(\Omega;\R^m)$. Set $\Omega_A=I_1\times\Omega'_\tau$ and $\Omega_B=I_2\times\Omega'_\tau.$ Clearly, (a) follows from (1). As $g(x)=u(x_1)=u(x)$ for $x\in \Omega_A\cup\Omega_B$, we have
\[
\nabla\Phi g(x)=\left\{\begin{array}{ll}
                         (1-\lambda)a\otimes e_1 & \mbox{if $x\in \Omega_A$}, \\
                       -\lambda a\otimes e_1 & \mbox{if $x\in \Omega_B$};
                       \end{array}
 \right.
\]
hence (c) holds. Also,
\[
\big||\Omega_A|-\lambda|\Omega|\big|=\big||\Omega_A|-\lambda\big|=\big||I_1||\Omega'_\tau|-\lambda\big|=\big||I_1|-|I_1||\Omega'\setminus\Omega'_\tau|-\lambda\big|<\tau,
\]
and likewise
\[
\big||\Omega_B|-(1-\lambda)|\Omega|\big|<\tau;
\]
so (d) is satisfied.
Note that for $i=1,\cdots,m,$
\[
\begin{split}
\Phi^i g & = \Phi^i(\eta u) = \sum_{1\le k\le m,\,1\le l\le n}a^i_{kl}(\eta u^k)_{x_l}=\eta\Phi^i  u+\sum_{1\le k\le m,\,1\le l\le n}a^i_{kl}\eta_{x_l} u^k\\
& = \eta\Phi^i  u+ u^2\sum_{l=1}^n a^i_{2l}\eta_{x_l} =\eta a^i_{21}u^2_{x_1} + u^2\sum_{l=1}^n a^i_{2l}\eta_{x_l}.
\end{split}
\]
So
\[
\|\Phi g\|_{L^\infty(\Omega)}\le C\max\{\delta,\delta\tau^{-1}\}<\tau
\]
if $\delta>0$ is chosen small enough; so (e) holds.
Next, for $i=1,\cdots,m$ and $j=1,\cdots,n,$
\[
\partial_{x_j}\Phi^i g=\eta_{x_j}a^i_{21}u^2_{x_1}+\eta\partial_{x_j}\Phi^i u + u^2_{x_j}\sum_{l=1}^n a^i_{2l}\eta_{x_l} + u^2\sum_{l=1}^n a^i_{2l}\eta_{x_l x_j};
\]
hence from (\ref{rank-1-9}),
\[
\dist(\nabla\Phi g,[-\lambda a\otimes e_1,(1-\lambda) a\otimes e_1])\le C\max\{\delta \tau^{-1},\delta\tau^{-2}\}<\tau\;\;\mbox{in $\Omega$}
\]
if $\delta$ is sufficiently small. Thus (b) is fulfilled.

\textbf{(Case 2):} Assume that $L_{i1}=0$ for all $i=2,\cdots, m$, that is,
\begin{equation}\label{rank-1-3}
L=\begin{pmatrix} L_{11} & L_{12} & \cdots & L_{1n}\\
                  0  & L_{22} & \cdots & L_{2n}\\
                  \vdots  & \vdots & \ddots & \vdots\\
                  0  & L_{m2} & \cdots & L_{mn} \end{pmatrix}\in\M^{m\times n}
\end{equation}
and that
\[
A-B=a\otimes e_1\;\;\mbox{for some nonzero vector $a\in\R^m$};
\]
then by (\ref{rank-1-1}), we have $L_{11}\ne 0.$

Set
\[
\hat L=\begin{pmatrix} L_{22} & \cdots & L_{2n}\\
                  \vdots & \ddots & \vdots\\
                  L_{m2} & \cdots & L_{mn} \end{pmatrix}\in\M^{(m-1)\times (n-1)}.
\]
As $L_{11}\ne 0$ and $\rank(L)=r$, we must have $\rank(\hat L)=r-1.$ Using the singular value decomposition theorem, there exist two matrices $\hat U\in O(m-1)$ and  $\hat V\in O(n-1)$ such that
\begin{equation}\label{rank-1-4}
\hat U^T\hat L\hat V=\mathrm{diag}(\sigma_2,\cdots,\sigma_r,0,\cdots,0)\in\M^{(m-1)\times (n-1)},
\end{equation}
where $\sigma_2,\cdots,\sigma_r$ are the positive singular values of $\hat L.$ Define
\begin{equation}\label{rank-1-2}
U=\begin{pmatrix} 1 & 0\\
                  0 & \hat U\end{pmatrix}\in O(m),\;\;
V=\begin{pmatrix} 1 & 0\\
                  0 & \hat V\end{pmatrix}\in O(n).
\end{equation}
Let $L'=U^T LV$, $A'=U^T AV$, and $B'=U^T BV$. Let $\mathcal{L}':\M^{m\times n}\to \R$ be the linear map given by
\[
\mathcal{L}'(\xi')=\sum_{1\le i\le m,\,1\le j\le n}L'_{ij}\xi'_{ij}\quad \forall \xi'\in\M^{m\times n}.
\]
Then, from (\ref{rank-1-3}), (\ref{rank-1-4}) and (\ref{rank-1-2}), it is straightforward to check the following:
\[
\left\{
\begin{array}{l}
  \mbox{$A'-B'=a'\otimes e_1$ for some nonzero vector $a'\in\R^m$,} \\
  \mbox{$L' e_1\neq 0$, $\mathcal L'(A)=\mathcal L'(B)$, and} \\
  \mbox{$L'$ is of the form (\ref{rank-1-5}) in Case 1 with $\mathrm{rank}(L')=r$.}
\end{array}\right.
\]
Thus we can apply the result of Case 1 to find a linear operator $\Phi':C^1(\R^n;\R^m)\to C^0(\R^n;\R^m)$ satisfying the following:

(1') For any open set $\Omega'\subset\R^n$,
\[
\Phi' v'\in C^{k-1}(\Omega';\R^m)\;\;\mbox{whenever}\;\; k\in\N\;\;\mbox{and}\;\;v'\in C^{k}(\Omega';\R^m)
\]
and
\[
\mathcal{L'}(\nabla\Phi' v')=0 \;\;\mbox{in}\;\;\Omega'\;\;\mbox{for all}\;\;v'\in C^2(\Omega';\R^m).
\]

(2') Let $\Omega'\subset\R^n$ be any bounded domain. For each $\tau>0$, there exist a function $g'=g'_\tau\in  C^{\infty}_{c}(\Omega';\R^m)$ and two disjoint open sets $\Omega'_{A'},\Omega'_{B'}\subset\subset\Omega'$ such that
\begin{itemize}
\item[(a')] $\Phi' g'\in C^\infty_c(\Omega';\R^m)$,
\item[(b')] $\dist(\nabla\Phi' g',[-\lambda(A'-B'),(1-\lambda)(A'-B')])<\tau$ in $\Omega'$,
\item[(c')] $\nabla \Phi' g'(x)= \left\{\begin{array}{ll}
                                (1-\lambda)(A'-B') & \mbox{$\forall x\in\Omega'_{A'}$}, \\
                                -\lambda(A'-B') & \mbox{$\forall x\in\Omega'_{B'}$},
                              \end{array}\right.$
\item[(d')] $\big||\Omega'_{A'}|-\lambda|\Omega'|\big|<\tau$, $\big||\Omega'_{B'}|-(1-\lambda)|\Omega'|\big|<\tau$,
\item[(e')] $\|\Phi' g'\|_{L^\infty(\Omega')}<\tau$.
\end{itemize}

For $v\in C^1(\R^n;\R^m)$, let $v'\in C^1(\R^n;\R^m)$ be defined by $v'(y)=U^T v(Vy)$ for $y\in\R^n$. We define $\Phi v(x)=U\Phi' v'(V^T x)$ for $x\in\R^n$, so that $\Phi v\in C^0(\R^n;\R^m).$ Then it is straightforward to check  that properties (1') and (2') of  $\Phi'$ imply respective properties (1) and (2) of  the linear operator $\Phi:C^1(\R^n;\R^m)\to C^0(\R^n;\R^m)$.

\textbf{(Case 3):} Finally, we consider the general case that $A$, $B$ and $L$ are as in the statement of the theorem. As $|b|=1$, there exists an $R\in O(n)$ such that $R^T b=e_1\in\R^n$. Also there exists a symmetric (Householder) matrix $P\in O(m)$ such that the matrix $L':=PLR$ has the first column parallel to $e_1\in\R^m$. Let
\[
A'=PAR\;\;\mbox{and}\;\; B'=PBR.
\]
Then $A'-B'=a'\otimes e_1$, where $a'=Pa\ne 0$. Note also that $L'e_1=PLRR^tb=PLb\ne 0$. Define $\mathcal L'(\xi')=\sum_{i,j}L'_{ij}\xi'_{ij}\;(\xi'\in\M^{m\times n})$; then $\mathcal L'(A')=\mathcal L(A)=\mathcal L(B)=\mathcal L'(B')$. Thus by the result of  Case 2, there exists a linear operator $\Phi':C^1(\R^n;\R^m)\to C^0(\R^n;\R^m)$ satisfying (1') and (2') above.

For $v\in C^1(\R^n;\R^m)$, let $v'\in C^1(\R^n;\R^m)$ be defined by $v'(y)=Pv(Ry)$ for $y\in\R^n$, and define $\Phi v(x)=P\Phi'v'(R^Tx)\in C^0(\R^n;\R^m)$. Then it is easy to check that the linear operator $\Phi:C^1(\R^n;\R^m)\to C^0(\R^n;\R^m)$ satisfies (1) and (2) by (1') and (2') similarly as  in  Case 2.


\begin{thebibliography}{11}




\bibitem{AH}
R. Aumann and S. Hart, {\em Bi-convexity and bi-martingales}, Israel J. Math., {\bf 54} (2) (1986), 159--180.

\bibitem{BJ}
J.M. Ball and R.D. James,  {\em Fine phase mixtures as minimizers of energy},  {Arch. Rational Mech. Anal.}, {\bf 100} (1) (1987), 13--52.

\bibitem{BJ1}
J.M. Ball and R.D. James,  {\em Proposed experimental tests of a theory of fine structure and the two-well problem},  {Phil. Trans. Roy. Soc. London A}, {\bf 338} (1992), 389--450.








\bibitem{Ca}
E. Casadio-Tarabusi, {\em An algebraic characterization of quasi-convex functions},  Ricerche Mat., {\bf 42} (1) (1993),  11--24.


\bibitem{CK}
M. Chipot and D. Kinderlehrer,  {\em Equilibrium configurations of crystals}, Arch. Rational Mech. Anal., {\bf 103} (3) (1988), 237--277.

\bibitem{CK1}
M. Chleb\'ik and B. Kirchheim, {\em Rigidity for the four gradient problem}, J. Reine Angew. Math., {\bf 551} (2002), 1--9.


\bibitem{CDK}
S. Conti, G. Dolzmann and B. Kirchheim, {\em Existence of Lipschitz minimizers for the three-well problem in solid-solid phase transitions},  Ann. Inst. H. Poincar\'e Anal. Non Lin\'eaire, {\bf 24} (6) (2007), 953--962.

\bibitem{CFG}
D. Cordoba, D. Faraco and  F. Gancedo, {\em Lack of uniqueness for weak solutions of the incompressible porous media equation},  Arch. Ration. Mech. Anal., {\bf 200} (3) (2011), 725--746.

\bibitem{Da}
B. Dacorogna, ``Direct methods in the calculus of variations," Second edition. Applied Mathematical Sciences, 78. Springer, New York, 2008.

\bibitem{DM1}
B. Dacorogna and P. Marcellini, ``Implicit partial differential equations," Progress in Nonlinear Differential Equations and their Applications, 37. Birkh\"auser Boston, Inc., Boston, MA, 1999.

\bibitem{DT}
B. Dacorogna and C. Tanteri, {\em  Implicit partial differential equations and the constraints of nonlinear elasticity,} J. Math. Pures Appl. (9), {\bf 81} (4) (2002), 311--341.





\bibitem{DS}
C. De Lellis and L. Sz\'ekelyhidi Jr., {\em  The Euler equations as a differential inclusion,} Ann. of Math. (2), {\bf 170} (3) (2009), 1417--1436.











\bibitem{Gr}
M. Gromov, {\em Convex integration of differential relations. I},  Izv. Akad. Nauk SSSR Ser. Mat., {\bf 37} (1973), 329--343.











\bibitem{KK1}
S. Kim and Y. Koh, {\em Weak solutions for one-dimensional non-convex elastodynamics}, Preprint.


\bibitem{KY1}
S. Kim and B. Yan, {\em Convex integration and infinitely many weak solutions to the Perona-Malik equation in all dimensions}, SIAM J. Math. Anal., {\bf 47} (4) (2015), 2770-2794.

\bibitem{KY2}
S. Kim and B. Yan, {\em On Lipschitz solutions for some  forward-backward parabolic  equations}, Preprint.

\bibitem{KY3}
S. Kim and B. Yan, {\em On Lipschitz solutions for some  forward-backward parabolic  equations. II: The case against Fourier}, Preprint.



\bibitem{Ku} 	
N.H. Kuiper, {\em On $C^1$-isometric embeddings. I},  Nederl. Akad. Wetensch. Proc. Ser. A., {\bf  58} (1955), 545--556.





\bibitem{MRS}
S. M\"uller, M.O. Rieger and V. \v Sver\'ak, {\em Parabolic systems with nowhere smooth solutions}, Arch. Rational Mech. Anal., {\bf 177} (1) (2005),  1--20.

\bibitem{MSv1}
S. M\"uller and V. \v Sver\'ak, {\em Convex integration with constraints and applications to phase transitions and partial differential equations}, J. Eur. Math. Soc. (JEMS), {\bf 1} (4) (1999), 393--422.

\bibitem{MSv2}
S. M\"uller and V. \v Sver\'ak, {\em Convex integration for Lipschitz mappings and counterexamples to regularity}, Ann. of Math. (2), {\bf 157} (3) (2003), 715--742.



\bibitem{Na1}
J. Nash, {\em $C^1$ isometric imbeddings}, Ann. of Math. (2), {\bf 60} (1954), 383--396.

\bibitem{Pe1}
P. Pedregal, {\em Laminates and microstructure}, Europ. J. Appl. Math., {\bf 4} (2) (1993),  121--149.





\bibitem{Po}
L. Poggiolini, {\em Implicit pdes with a linear constraint}, Ricerche Mat.,  {\bf 52} (2)  (2003),   217--230.

\bibitem{Po1}
W. Pompe, {\em Explicit construction of piecewise affine mappings with constraints,} Bull. Pol. Acad. Sci. Math., {\bf 58} (3) (2010),  209--220.



\bibitem{Sc}
V. Scheffer, ``Regularity and irregularity of solutions to nonlinear second-order elliptic systems of partial differential equations and inequalities," Dissertation, Princeton University, 1974.




\bibitem{Sy}
R. Shvydkoy, {\em Convex integration for a class of active scalar equations}, J. Amer. Math. Soc., {\bf 24} (4) (2011), 1159--1174.




\bibitem{Ta1}
L. Tartar, ``Some remarks on separately convex functions, in: Microstructure and
Phase Transitions," IMA Vol. Math. Appl. 54 (D. Kinderlehrer, R. D. James,
M. Luskin and J. L. Ericksen, eds.), Springer-Verlag, New York, 1993, pp. 191--204.










\end{thebibliography}
\end{document}